\documentclass[preprint,3p,onecolumn]{elsarticle}

\usepackage{amssymb,amsmath,amsthm}
\usepackage[UKenglish]{babel}
\usepackage{graphicx}
\usepackage{graphics}
\usepackage{bm}
\usepackage{color}
\usepackage{bbm}

\newtheorem{theorem}{Theorem}

\newtheorem{lemma}{Lemma}
\newtheorem{definition}{Definition}
\newtheorem{rem}{Remark}

\newtheorem{assumptionp}{Assumption}

\newcommand\f{\frac}
\newcommand{\G}{\mathcal{G}}
\newcommand{\svet}{\texttt{s}}
\newcommand{\ivet}{\texttt{i}}
\newcommand{\rvet}{\texttt{r}}
\newcommand{\evet}{\texttt{e}}
\newcommand\rr{{\mathbb R}}
\DeclareMathOperator*{\esssup}{ess\,sup}

\newcommand\nn{{\mathbb N}}

\newcommand{\vli}{{\boldsymbol{\imath}}}

\newcommand{\vone}{{\boldsymbol{1}}}
\newcommand{\vzero}{{\boldsymbol{0}}}

\DeclareMathOperator{\Diag}{Diag}

\newcommand{\vS}{{\bf S}}
\newcommand{\vE}{{\bf E}}
\newcommand{\vI}{{\bf I}}
\newcommand{\vR}{{\bf R}}
\newcommand{\vV}{{\bf V}}
\newcommand{\vs}{{\bf s}}
\newcommand{\ve}{{\bf e}}
\newcommand{\vr}{{\bf r}}
\begin{document}
\begin{frontmatter}
\author{Giovanni Naldi - UNIMI}
\author{Giuseppe Patan\`{e} - CNR}


\address{UNIMI, Department of Environmental Science and Policy, via Celoria 2, 20133 Milano, Italy}
\address{CNR-IMATI, Via De Marini 6, 16149 Genova, Italy}
%
\title{A Graph-Based Modelling of Epidemics:\\ Properties, Simulation, and Continuum Limit}
\begin{abstract}
This work is concerned with epidemiological models defined on networks, which highlight the prominent role of the social contact network of a given population in the spread of infectious diseases. In particular, we address the modelling and analysis of very large networks. As a basic epidemiological model, we focus on a SEIR (Susceptible-Exposed-Infective-Removed) model governing the behaviour of infectious disease among a population of individuals, which is partitioned into sub-populations. We study the long-time behaviour of the dynamic for this model, also taking into account the heterogeneity of the infections and the social network. By relying on the theory of graphons, we address the natural question of the large population limit and investigate the behaviour of the model as the size of the network tends to infinitely. After establishing the existence and uniqueness of solutions to the selected models, we discuss the use of the graphon-based limit model as a generative model for a network with particular statistical properties related to the distribution of connections. We also provide some preliminary numerical tests.
\end{abstract}
\end{frontmatter}
\section{Introduction\label{sec:intro}}
The modelling of infectious disease transmission has a long history in mathematical biology, but in recent years it has gained an increasing influence on the theory and practice of disease management and control.~\cite{AM1991,Brauer2017,Dieketal2013,He2000,Keelingetal2008}. Most epidemiological models are compartmental models, with the population divided into classes and with assumptions about the rate of transfer from one class to another. In this paper we consider as a basic model the \emph{Susceptible-Exposed-Infectious-Removed} (SEIR) \emph{model} describing the disease transmission and the rate of infected individuals. In fact, the SEIR model, and some of its modifications, is widely considered in the literature and in applications as a starting point for describing the spread of various diseases, for example, tuberculosis~\cite{Tuber17}, measles~\cite{measles20}, MERS~\cite{Mers03}, and recently COVID-19~\cite{Covid19:21}. The basic idea of the SEIR model~\cite{He2000,Keelingetal2008} is to describe the number of infected and recovered individuals based on the number of contacts, probability of disease transmission, incubation period, recovery, and mortality rate. Since the model focuses on a very short time with respect to demographic dynamics, it postulates that births and natural (i.e., not connected with epidemics) deaths balance each other.

While, in some cases, the behaviour of average quantities of a (large) with respect to the time is sufficient to provide useful insight into the spread of the epidemics from the available data, the spatial component of many transmission systems has been recognised to be of pivotal importance. Due to this, spatially heterogeneous features must be included in the model to properly represent the transmission pattern. Then, a reasonable hypothesis about the phenomena may consider that the spatial aspects of transmission heavily influence the aggregation characteristic of the epidemic: we need hence to investigate data by using models that include such spatial connections. For example, the understanding of human mobility and the development of qualitative and quantitative theories is of key importance for the modelling and comprehension of human infectious disease dynamics, on geographical scales of different sizes. Also for the spread of infectious diseases in livestock comprehensive information on livestock movements, cattle movement, and contacts is required to devise appropriate disease control strategies. Understanding contact risk when herds mix extensively, and where different pathogens can be transmitted at different spatial and temporal scales, remains a major challenge~\cite{Livestock21}. For example, using data related to cattle movements and focusing on the geographical distribution of these movements is possible to improve the analysis of the spread of epizootic diseases~\cite{Epizotic15}.

To introduce spatial heterogeneity we consider metapopulation-based models, where the population is partitioned into large, spatially segregated sub-populations. A similar approach could be used in a more general way, irrespective of the biological interpretation: different ages, small interacting communities, and so on~\cite{Eubanketal2004,Vespignanietal2015,Newman2002}. This framework has traditionally provided an attractive approach to incorporating more realistic contact structures into epidemic models since it often preserves analytic tractability but also captures the most salient structural inhomogeneities in contact patterns. Understanding the dynamics of coupled systems on graphs modelling connectivity in real-life systems can be quite challenging. On the other hand, one is often interested in the analysis of average quantities over large networks~\cite{LS07}. For the analysis of the behaviour of the correspondent coupled dynamical systems, we consider the continuum limit where the dynamics on a sequence of large graphs (networks) is replaced by an evolution equation on a continuous spatial domain. Many nontrivial graph sequences have relatively simple limits, described by symmetric measurable functions on a unit square, called \emph{graphons}. In particular, information on the asymptotic distribution of social contacts can be encoded in the theory of graphons to obtain continuum models that approximate the dynamics over graphs with a large number of nodes.

The paper is organised as follows. In Sect.~\ref{sec:RELATED-WORK}, we briefly previous work on graphons and in Sect.~\ref{sec:PF}, we introduce the SEIR model on a graph and its main properties. Our main goal is to understand the interplay between the transmission of the disease and the distribution of contacts. Our model does not take into account some relevant factors, like the presence of asymptomatic individuals that have a prominent role in disease transmission. We point out that these model improvements could be analysed with methods similar to those introduced in this work. In Sect.~\ref{sec:spectra}, we estimate the spectrum of the weighted adjacency matrix involved in the dynamical model and show that the spectrum has a prominent role in the asymptotic behaviour of the epidemic. In Sect.~\ref{sec:graphon}, we derive the continuum limit of compartmental epidemiological models on graphs based on a suitable asymptotic procedure, which relies on the notion of graphon. Then, we investigate the relations between the solutions of the graphon-based asymptotics and the SEIR dynamical system on a suitable graph. Finally, we provide some numerical tests. In the numerical experiments and following the detailed analysis in~\cite{Beraudetal2015}, the number of social contacts in the population is described by a Gamma distribution. In Sect.~\ref{sec:majhead}, we discuss conclusions and future work.

\section{Basic facts about graphons\label{sec:RELATED-WORK}}
Graphon theory was introduced and developed in recent years by Lov\'{a}sz, Szegedy, Borgs, Chayes, S\'{o}s, and Vesztergombi among others~\cite{LO12,BCCZ18,BCCZ19,BCLSV06,BCLSV12}. A graph~$G=(V,E)$ is a set of nodes (vertices)~$V(G)\,=\,\{ 1,\, 2,\, \ldots ,\,n\}$, usually~$n=\vert V(G) \vert$ where~$\vert \cdot \vert~$ denotes the cardinality of a set, and a set of edges~$E(G)\subseteq V\times V$ between the vertices. In what follows the graphs will be simple, without loops or multiple edges, and finite unless otherwise specified. Weights (real numbers) will be given to the edges of a graph to make it an edge-weighted graph. Moreover, we assume that the graph is undirected, i.e., we identify the edges~$(i,j)$ and~$(j,i)$. Let~$\tilde A=(\tilde a_{ij})$ be the adjacency matrix of the graph:
\[
\tilde a_{ij}=\left\{
\begin{array}{l}
 1\,\,\,\text{if }(i,j)\text{ is an edge,}\\
 0\,\,\,\text{otherwise.}
\end{array}
\right.
\]
Let~$G$ and~$H$ be graphs, a map~$\phi$ from~$V(H)$ to~$V(G)$ is a homomorphism if it preserves edge adjacency, that is if for every edge~$(i,j)$ in~$E(H)$,~$(\phi (i),\, \phi(j))$ is an edge in~$E(G)$. Denote by~$hom(H,G)$ the number of homomorphisms from~$H$ to~$G$. For example for~$H$ with one node and no edge~$hom(H,G)=|V(G)|$, instead when~$H$ has two nodes and one edge~$hom(H,G)=|E(G)|$, or when~$H$ has three vertices and~$E(H)=\{(1,2),\, (2,3),\,(3,1)\}$,~$hom(H,G)$ is~$6$ times the number of triangles in~$G$. Normalising by the total number of possible maps, we get the density of homomorphisms from~$H$ to~$G$, \mbox{$t(H,G)=\frac{hom(H,G)}{\vert V(G) \vert ^{\vert V((H)\vert }}$}. It is defined that a sequence of simple graphs~$\{G_n\}_{n\in\mathbb{N}}$ is convergent if~$G_n$ become more and more similar as~$n$ goes to infinity.
\begin{definition} \label{left_conv}
A sequence of graphs~$\{G_n\}_{n\in\mathbb{N}}$ is said to be left convergent if the sequences~$t(H,G_n)$ converges for~$n\rightarrow\infty$, for every simple graph~$H$.
\end{definition}
The breakthrough results of Lov\'{a}sz and Szegedy~\cite{LO12} is that a (left) convergent undirected graph sequence has a limit object, which can be represented as a measurable function.
\begin{definition} \label{def:graphon}
A {\it graphon} is a bounded measurable functions~$W:\,[0,1]^2 \rightarrow \mathbb{R}$ that satisfy~$W(x,y)=W(y,x)$ for all~$x,\,y\in [0,1]$.
\end{definition}
Let~$\cal{W}$ denote the space of all the graphon and~$\cal{W}_0$ is the set of graphons with values in~$[0,1]$. In our models, we will consider, due to the meaning of the connections between different sub-population only graphons in~$\cal{W}_0$. We give below the representation Theorem of Lov\'asz and Szegedy~\cite{LS06}.
\begin{theorem}\label{Teo:LS}
For every (left) convergent sequence of simple graphs, there is~$W\in \cal{W}_0$ such that
\begin{equation}\label{eq:repres}
t(H,G_n) \rightarrow t(H,W) := \int_{I^{\vert V(H)\vert}}\,\,\prod_{(i,j=\in E(H)}\,W(x_i,x_j)dx
\end{equation}
for every simple graph~$H$, where~$I=[0,1]$. Moreover, for every~$W\in\cal{W}_0$ there is a sequence of graphs~$\{ G_n\}_{\mathbb{N}}$ satisfying \eqref{eq:repres}.
\end{theorem}
Heuristically the intuition behind this definition is that the interval~$[0,1]$ represents a sort of continuum of vertices, serving as locations or indices, and~$W(x,y)$ denotes the probability of having an edge between vertices~$x$ and~$y$. It is possible to represent a graph~$G$ with a graphon from a suitable family of graphons. If~$\vert V(G) \vert=n$ and choosing any labelling of the nodes, we divide the interval~$[0,1]$ in~$n$ sub-intervals~$I_j^n$,~$j=1,2,\ldots, n$ as
\begin{equation}\label{def:partition}
I_1^n=\left[ 0,\frac{1}{n} \right];\,\, I_j^n=\left( \frac{j-1}{n},\,\frac{j}{n} \right]\,\,j=2,\ldots ,n;
\end{equation}
and define the piecewise constant graphon~$W_G$ by
\begin{equation}\label{eq:graph-on}
W_G(x,y)=a_{ij}\,\,\text{if}\,\, (x,y)\in I_i^n\times I_j^n.
\end{equation}
Then,~$W_G$ is a kind of functional representation of the adjacency matrix~$A$ of the graph~$G$. We point out that different labelling yields different graphon associated with the same graph.

It remains to formalise the convergence of the associated graphons to a graph and understand in what sense the space of graphons completes the space of finite graphs. To do this, we define a metric on the set of labelled graphons. In this work, we will use the set of graphons~$\cal{W}_0$ but the same definitions can be extended to the general graphon space~$\cal{W}$. By analogy with the matrix cut norm, we can define the cut norm of~$W\in \cal{W}_0$ by setting
\begin{equation}
		\label{cut1}
		\|W\|_{\Box}:=\sup_{S,T\subseteq[0,1]}\left\vert\int_{S\times T}W(x,y)\,dxdy\right\vert ,
	\end{equation}
where the supremum is taken over all measurable subsets~$S$ and~$T$. The notion of cut-norm was first introduced by Frieze and Kannan~\cite{FK99}, and its fundamental importance for the theory of graphons was recently unveiled in several papers, see in particular~\cite{BCCZ18,BCCZ19,BCLSV06,BCLSV08,BCLSV12} and the references therein. In particular, we point out that the notion of cut-norm for the graphons associated with a sequence of graphs is connected with the notion of the (left) convergence for graph sequences, see~\cite[Theorem 3.8]{BCLSV08}. Very loosely speaking, a sequence of graphs is left convergent if the local structure of the graphs is asymptotically stable, in the sense that asymptotically the graphs contain the ``same number of copies'' of any given subgraph. A drawback of the notion of the cut norm is related to the fact that for two graphs~$G$,~$G^\prime$ which are the same but with different node labelling the corresponding adjacency matrices are not the same. Then, we have~$\|W_G - W_{G^\prime}\|_{\Box} \neq 0$. This observation leads to the definition of the cut distance.
\begin{definition} \label{cut:distance}
Let~$\cal{M}$ the set of all measure-preserving bijections~$\phi : [0,1] \rightarrow [0,1]$, the cut distance between~$W,\, U\in\cal{W}$ is defined as
\begin{equation*}
\left\{
\begin{array}{l}
\delta_{\Box}(W,U )=\inf_{\phi ,\psi\in\cal{M}}\,\| W^\phi - U^\psi \|_{\Box},\\
W^\phi (x,y)=W(\phi (x),\phi (y)),\quad U^\psi (x,y)=W(\psi (x),\psi (y)).
\end{array}
\right.
\end{equation*}
\end{definition}
Since the cut metric of two different graphons can be zero, strictly speaking, it is not a metric. By identifying graphons~$U$ and~$W$ for which~$\delta_{\Box}(W,U )=0$, we can construct a metric space~$(\cal{W}_0,\delta_{\Box})$ which is compact under this identification~\cite{LS07}. For a sequence of graphons~$\{ W_n \}_{n\in \mathbb{N}}$ in~$\cal{W}_0$ it is possible to prove~\cite{BCLSV08} that the convergence of~$t(H,W_n)$, as~$n\rightarrow \infty$, for any simple graph~$H$ is equivalent to requesting that ~$\{ W_n \}_{n\in \mathbb{N}}$ is a Cauchy sequence for~$\delta_{\Box}$ or that exists~$W\in \cal{W}_0$ such that~$\delta_{\Box} (W_n,W) \rightarrow 0$. Using the association between simple graphs and graphons we have the following result~\cite{BCLSV08}.
\begin{theorem}\label{Teo:LS_conv}
Let~$\{G_n\}_{n\in\mathbb{N}}$ be any (left)convergence sequence of simple graphs. There exists a graphon~$W\in\cal{W}_0$ such that \mbox{$\lim_{n\rightarrow \infty}\, \delta_{\Box}(W_{G_n},W)=0$}. Conversely, every~$W\in\cal{W}_0$ is a limit in the cut metric of a convergent sequence of simple graphs.
\end{theorem}
The limit in Theorem~\ref{Teo:LS_conv} is unique up to the identification of graphons with a cut distance equal to zero. Graphons are then the limit objects of converging graph sequences. We will construct and analyse the limit of the dynamical systems on a sequence of simple graphs by using the graphon structure. In particular, we characterise the asymptotic behaviour of the corresponding sequence of SEIR models defined on the simple graphs~$G_n$. Pioneering works about dynamical systems and graphon were done by G S. Medvedev for the Kurtamoto model~\cite{Med19} and nonlinear heat equation~\cite{Med14}. As compared to~\cite{Med14,Med19} we handle the time-dependent case, the biological characteristic of the epidemics and the contact dynamics can change with respect to time. Moreover, we consider not only a scalar dynamical system on the single node but a vectorial one and, from the technical hypothesis, we do not require that~$W\in L^2([0,1]^2)$. In recent literature it is possible to find related works using graphons to study epidemics spread, e.g.~\cite{SHUANG2019,VIZUETE2020,AURELI2022}.

\section{A prototype problem: SEIR model on graphs\label{sec:PF}}
When investigating the transmission of infectious diseases, the analysis of the average behaviour of a large population is sufficient to provide useful insight and extract valuable information from the input data. However, the importance of the spatial component of many transmission systems is being increasingly recognised~\cite{Vespignanietal2015}. The main approaches for spatial models concern different scales: an individual-based simulation, a meta-population model, or a network model. Individual-based models explicitly represent every individual and usually assume a variable probability that any infectious host can infect any other susceptible host. Then the model should be able to account for the states of all~$N$ individuals in the population in an independent manner, and at the same time, it should allow for arbitrary interactions among them. The analysis of these models is a difficult task the computational cost of numerical simulations is very onerous, and the extraction of the collective behaviours is very complex. Conversely, in the meta-population models the number of individuals at different space locations is in some states. These models often assume that each location is connected to others, with possible variable strengths of connection.

To describe a mathematical model for the spread of infectious disease, one has to make some assumptions about the disease transmission. We consider here, as a basic model, the SEIR compartmental model where individuals are classified into different population groups based on the infection status. The model tracks the number of people in each of the following categories: \emph{Susceptibles} (individuals that may become infected), \emph{Exposed} (individuals that have been infected with a pathogen, but due to the pathogen incubation period, are not yet infectious), \emph{Infectious} (individuals that are infected with a pathogen and may transmit it to others), and \emph{Recovered} (individual that is either no longer infectious or has been ``removed'' from the population). Indeed, we consider diseases with a latent phase during which the individual is infected but not yet infectious: a recent example of the application of this type of model is the description of the transmission of the COVID-19 disease~\cite{Heetal2020}.

In the scalar case, the total (initial) population,~$N$, is categorised into four classes, namely,~$s(t)$ (susceptible),~$e(t)$ (exposed),~$i(t)$ (infected-infectious), and~$r(t)$ (recovered), where~$t$ is the time variable. The basic assumption for the scalar model is the homogeneously mixing population hypothesis which, roughly speaking, means that a given infectious individual may transmit the disease to any susceptible individual at the same rate. Also, one postulates that all the individuals in the population have the same chances of interacting with each other. People move from \texttt{S} to \texttt{E} based on the number of contacts with \texttt{I} individuals, multiplied by the probability of infection~$\beta$, where~$\beta i(t)/N$ is the average number of contacts with infection per unit time of one susceptible person. The other processes taking place at time~$t$ are: the exposed \texttt{E} become infectious \texttt{I} with a rate~$\mu$ and the infectious recover \texttt{R} with a rate~$\gamma$. Recovered individuals do not flow back into the \texttt{S} class, as lifelong immunity is postulated. The fractions~$1/\mu$ and~$1/\gamma$ are the average disease incubation and infectious periods, respectively. We assume that the total population remain constant, i.e., ~$s(t)+e(t)+i(t)+r(t)=N$. Next, we consider the rescaled variables, which for simplicity we do not relabel, ~$s(t)/N \rightarrow s(t)$,~$e(t)/N \rightarrow e(t)$, ~$i(t)/N \rightarrow i(t)$,~$r(t)/N \rightarrow r(t)$. The ordinary differential equations (ODEs) governing the SEIR model are then
\[ \left\{
\begin{array}{l}
\dot{s}(t)=-\beta\, s(t)i(t),\qquad
\dot{e}(t)= \beta\, s(t)i(t) - \mu e(t),\\
\dot{i}(t)=\mu e(t) -\gamma i(t),\qquad
\dot{r}(t)=\gamma i(t).
\end{array} \right.
\]
We consider a meta-population model and a network that is represented by a \emph{simple graph}~$\mathcal{G}=(V,E)$ with~$n$ vertices (nodes, regions, patches) and~$m$ edges (connections). Each edge is described by a couple of nodes~$(u,v)$,~$u,v\in V$. We order the nodes and label them with an integer index. We assume that the adjacency matrix of~$G$ is an irreducible matrix: there are no isolated and unreachable groups. 
%
In node~$j$ the corresponding sub--population possesses~$N_j$ individuals, and~$\sum_{j=1}^{n} N_j=N$. We assume that individuals can move to a different node, interact with people in that node, and then return to the original one. We denote by~$\hat{a}_{jk}$ the total amount of the subpopulation~$j$ that ``goes'' to node~$k$ and interacts with the people in that node. We call~$\hat{A}$ the matrix with entries~$\hat{a}_{jk}$, so that~$\sum_{k=1}^{n} \hat{a}_{jk}=N_j$,~$j=1,\dots,n$. Associated to~$\hat{A}$, let~$P^{out}$ the \textit{probability outgoing matrix} with entries~$p^{out}_{jk}$, where we denote by~$p^{out}_{jk}$ the probability (percentage) that the sub--population~$j$ ``goes'' to node~$k$. In addition, we denote by~$P^{in}$ the \textit{probability incoming matrix} with entries~$p^{in}_{jk}$, where~$p^{in}_{jk}$ is now the probability (percentage) of the sub--population in~$k$ that ``arrived'' from~$j$. Finally, let~$M_j = \sum_{k=1}^{n} \hat{a}_{kj}$ be the total amount of people arrived in node~$j=1,\dots,n$, so that~$\sum_{j=1}^{n} M_j=N$ again. Then, for any~$j=1,\dots,n$,~$\sum_{k=1}^{n} p^{out}_{jk}=\sum_{k=1}^{n} p^{in}_{jk}=1$. Moreover, we have
\[
\hat{A}=\Diag(N_1,\,N_2,\dots,\,N_n)P^{oout} = P^{in} \Diag(M_1,\,M_2,\dots,\,M_n)
\]
where~$\Diag(x_1,\,x_2,\dots,\,x_n)$ is the diagonal matrix with the vector~$(x_1,\,x_2,\dots,\,x_n)^T\in\mathbb{R}^n$ on the main diagonal.

Following the (SEIR model) four different discrete classes are considered for the statuses of individuals in each: susceptible, exposed, infectious, and recovered. Let~$S_j(t),\,E_j(t),\,I_j(t),\,R_j(t)$ the number of individuals in the node~$j$ at time~$t$,~$S_j(t)+E_j(t)+I_j(t)+R_j(t)=N_j$: we consider a time interval in which we can neglect demographics. Without any interaction with other nodes, within a deterministic approach of the compartmental models, with continuous time~$t$, the epidemic dynamics can be described by the system of differential equations in~\eqref{eq:seir:scalar}:
\begin{equation}\label{eq:seir:scalar}
\left\{
\begin{array}{lll}
\dot{S}_j(t)=-\lambda\, S_j(t), &
&\dot{E}_j(t)= \lambda\, S_j(t) - \mu E_j(t) \\
\dot{I}_j(t)=\mu E_j(t) -\gamma I_j(t),
& &\dot{R}_j(t)=\gamma I_j(t)
\end{array}
\right.
\end{equation}
where the parameter~$\lambda$ is the force of infection.

Concerning the behaviour of an epidemic,~$\lambda$ is the rate at which susceptible individuals become infected or exposed, and it is a function depending on the number of infectious individuals: it contains information about the interactions between individuals that concur with the infection transmission. If we suppose that the population of~$N_j$ individuals mix at random, meaning that all pairs of individuals have the same probability of interacting, the force of infection may be computed as:
\[
\begin{array}{ll}
\lambda & = \text{ trasmission rate} \\
 &\quad \times \text{ effective number of contacts per unit time}\\
 &\quad \times \text{ proportion of contacts infectious}\\
 & \sim \beta \frac{I_j}{N_j}
\end{array}
\]
where~$\beta$ is the infectious rate. Then the system state,
\begin{equation}\label{eq:seir:beta}
\left\{
\begin{array}{lll}
\dot{S}_j(t)=- \beta \frac{I_j}{N_j} S_j(t),
& &\dot{E}_j(t)= \beta \frac{I_j}{N_j} S_j(t) - \mu E_j(t), \\
\dot{I}_j(t)=\mu E_j(t) -\gamma I_j(t),
& &\dot{R}_j(t)=\gamma I_j(t)
\end{array}
\right.
\end{equation}
We denote by~$(s_j(t),\, e_j(t),\,\imath_j(t),\,r_j(t))$ the rescaled (percentage) quantities of susceptible, infected, removed at time~$t$ at the node~$j$ normalised to the number~$N_j$ of individuals associated with that node. Then, we obtain
\begin{equation}\label{eq:seir:rescaling}
\left\{
\begin{array}{lll}
\dot{s}_j(t)=- \beta \, \imath_j(t) s_j(t), &
&\dot{e}_j(t)= \beta \, \imath_j(t) s_j(t) - \mu e_j(t),\\
\dot{\imath}_j(t)=\mu e_j(t) -\gamma \, \imath_j(t),
& &\dot{r}_j(t)=\gamma \, \imath_j(t),
\end{array}
\right.
\end{equation}
where~$\dot{\imath}_j$ stands for the derivative of the function~$\imath_j$. 

Now, we take a node~$j$ that is connected with the other nodes as encoded in matrix~$\hat{A}$. Then~$S_j(t)$ can change due to the contribution of susceptible people that come from~$j$ reached an adjacent node~$k$ and met infectious people in that node. Then the contribution to~$\dot{S}_j$ due to the interactions in node~$k$ is given by the the~$p^{out}_{jk} S_j = \hat{a}_{jk} s_j$ susceptible people that met a population in node~$k$ with a proportion of infectious people given by
\[
\frac{\#\{\text{infectious people in node }k\}}{\#\{\text{total people in node }k\}}
=
\frac{\sum_{l=1}^n p^{out}_{lk} I_l }{\sum_{l=1}^n \hat{a}_{lk}} =
\sum_{l=1}^n p^{in}_{lk} \imath_l .
\]

Let the vectors~${\bf{X}}(t)=(x_1(t),\,x_2(t),\,\dots,\,x_n(t))^T \in \mathbb{R}^n$, with~$X=S,\,E,\,I,\,R$, the \emph{$SEIR$ model on the graph~$G$} is the following
\begin{equation}\label{eq_seir_graph}
\left\{
\begin{array}{lll}
\dot{\vS }(t)=- \beta \Diag (\vS (t)) \hat{B} \vI(t),
& &\dot{\vE }(t)= \beta \Diag (\vS (t)) \hat{B} \vI(t) - \mu \vE (t), \\
\dot{\vI }(t)=\mu \vE (t) -\gamma \vI (t),
& &\dot{\vR }(t)=\gamma \vI,
\end{array}
\right.
\end{equation}
where~$\hat{B} = P^{out} \Diag(M_1,\ldots,M_n)^{-1} {P^{out}}^T$, and after rescaling (obtained by a premultiplication with \mbox{$\Diag (N_1,N_2,\ldots,N_n)^{-1}$}),
\begin{equation}\label{eq:SIRgraph2}
\left\{
\begin{array}{ll}
\dot{\vs }(t)=- \beta \Diag (\vs (t)) A \vli(t),
&\dot{\ve}(t)= \beta \Diag (\vs (t)) A \vli(t) - \mu \ve (t),\\
\dot{\vli }(t)=\mu \ve (t) -\gamma \vli (t),
\dot{\vr }(t)=\gamma \vli (t),
\end{array}
\right.
\end{equation}
where~$A=P^{out}(P^{(in)})^T$. 
\begin{rem}
We have assumed that the parameter~$\beta$ is the same in all nodes. It is possible to easily introduce a different parameter for each node considering more heterogeneity in the model. Furthermore, both the parameters and the matrix~$A$ could be time-dependent: the analysis that we will make is easily adaptable to these cases as well. In particular, the change in the parameters, i.e. for~$\beta$, may represent a change in the biomedical aspects of the epidemics, while a change in the matrix~$A$ is a social activity.
\end{rem}
Owing to the classical Cauchy Lipschitz Picard Lindel\"of Theorem, the Cauchy problem obtained by coupling system~\eqref{eq:SIRgraph2} with the initial data
\begin{equation} \label{eq:idata}
 s_j(0) = s_{j0}, \quad e_j(0)= e_{j0}, \quad i_j(0) = i_{j0}, \quad r_j(0) = r_{j0}, \quad
 \text{for every~$j=1, \dots, n$}
\end{equation}
has a local in-time solution. Since~$s_j, e_j, i_j, r_j$ represent percentages of individuals, the modelling meaningful range is
\begin{equation}\label{eq:range}
 0 \leq s_j(t), e_j(t), i_j(t), r_j(t) \leq 1, \quad
 s_j(t)+ e_j(t)+ i_j(t)+ r_j(t) =1,
\end{equation}
for every~$ t \ge 0$ and every~$j=1, \dots, n$. The next result states that the solution of~\eqref{eq:SIRgraph2} does indeed attain values in the above range.
\begin{lemma}\label{l:bounds}
Assume that~$0 \leq s_{j0}, e_{j0}, i_{j0}, r_{j0} \leq 1$ and~$s_{j0}+ e_{j0}+ i_{j0}+ r_{j0}=1$, for every~$j=1, \dots, n$. Then, the solution of the Cauchy problem obtained by coupling~\eqref{eq:SIRgraph2} with~\eqref{eq:idata} is global in time and satisfies~\eqref{eq:range}.
\end{lemma}
\begin{proof}
To establish the second condition in~\eqref{eq:range}, it suffices to observe that~$\dot{s}_j (t)+ \dot{e}_j (t)+ \dot{\imath \oint}_j (t)+ \dot{r}_j (t)=0$ for every~$j=1, \dots, n$. We recall that, if the solution is bounded, then it can also be extended for every~$t \ge 0$. Indeed, it suffices to show that~$s_j(t), e_j(t), i_j(t), r_j(t) \ge 0$ for every~$t$ and every~$j=1, \dots, n$. By the continuous dependence of solutions of ODEs on parameters, it suffices to establish the same property for the family of perturbed systems
\begin{equation*}
\label{eq:SIRgraph3} \left\{
\begin{array}{l}
\dot{s}^\varepsilon_j(t)=-\beta\, s^\varepsilon_j(t) \sum_{k=1}^{n} a_{jk}i^\varepsilon_k(t) \\
\dot{e}^\varepsilon_j(t)=\beta\, s^\varepsilon_j(t) \sum_{k=1}^{n} a_{jk}i^\varepsilon_k(t) -\mu e^\varepsilon_j(t) + \varepsilon \\
\dot{i}_j(t)= \mu e^\varepsilon_j(t)-\gamma i^\varepsilon_j(t) + \varepsilon, \phantom{\sum_{k=1}^{n}}\\
\dot{r}_j(t)= \gamma i^\varepsilon_j(t),
\end{array} \right.
\end{equation*}
where~$\varepsilon>0$ is a parameter converging to~$0^+$. First, we observe that, if~$s^\varepsilon_j(t)=0$ for some~$t$, then~$\dot{s}^\varepsilon_j(t)=0$. By the uniqueness part of the Cauchy Lipschitz Picard Lindel\"of Theorem, this implies that~$s^\varepsilon_j(t) \equiv 0$ if~$s^\varepsilon_j(0)=0$ and hence that~$s^\varepsilon_j(t) > 0$ for every~$t$ if~$s_{j0} > 0$. Next, we set
$$
 \bar t_\varepsilon : = \min \{ t \ge 0: \; e^\varepsilon_j( t) =0 \; \text{or} \; i^\varepsilon_j (t) =0 \; \text{for some~$j =1, \dots, n$} \},
$$
and we separately consider the following cases. If there is~$j=1, \dots, n$ such that~$e^\varepsilon_j (\bar t_\varepsilon) =0$, then, since~$s^\varepsilon_j (\bar t), i^\varepsilon_j (\bar t_\varepsilon) \ge0$, we have~$\dot{e}^\varepsilon_j (\bar t_\varepsilon) \ge \varepsilon$ and hence~$\bar t_\varepsilon=0$ and~$e^\varepsilon_j >0$ in a left neighborhood of~$0$. If there is~$j=1, \dots, n$ such that~$i^\varepsilon_j (\bar t_\varepsilon) =0$, then~$\dot{i}^\varepsilon_j (\bar t_\varepsilon) \ge \varepsilon$ and hence~$\bar t_\varepsilon=0$ and~$i^\varepsilon_j >0$ in a left neighborhood of~$0$. This implies that~$0 \leq s^\varepsilon_j(t), e^\varepsilon_j(t), i^\varepsilon_j (t)$ for every~$t\ge 0$. Since~$\dot{r}_j^\varepsilon \ge 0$, then~$r^\varepsilon_j (t) \ge 0~$ for every~$t\ge 0$. By letting~$\varepsilon \to 0^+$, we conclude the proof of the lemma.
\end{proof}
In the following, we adopt the notation~$ \vone =[1\,1\,\ldots ,1]^{\top}$,~$ \vzero =[0,\,0,\,\ldots ,0]^{\top}$, and, for any vectors~$\mathbf{x},\, \mathbf{y}\in \mathbb{R}^n$ we write~$\mathbf{x}<<\mathbf{y}\,\Leftrightarrow x_i < y_i\,i=1,2,\ldots ,n$,~$\mathbf{x}\leq \mathbf{y} \,\Leftrightarrow x_i \leq y_i\,i=1,2,\ldots ,n$ and~$\mathbf{x}<\mathbf{y}$ if~$\mathbf{x}\leq \mathbf{y}$ but~$\mathbf{x}\neq \mathbf{y}$. The asymptotic behaviour of system~\eqref{eq:SIRgraph2} is described by the following Lemma.
\begin{lemma}\label{Teo:seir1}
Considering the SEIR system~\eqref{eq:SIRgraph2}, the equilibrium points are~$\mathbf{P}_\infty =(\mathbf{s}_\infty ,\, 0 ,\, 0 ,\,\vr_\infty )$. Also, consider the Cauchy problem obtained by coupling~\eqref{eq:SIRgraph2} with~\eqref{eq:idata}. Under the same assumptions as in Lemma~\ref{l:bounds}, every trajectory converges to an equilibrium point~$\mathbf{P}_\infty =(\vs_\infty ,\, 0 ,\, 0 ,\,\vr_\infty )$ with~$\vs_\infty^{\top} \vone +\vr_\infty^{\top} \vone = \vone$.
\end{lemma}
\begin{proof}
Let~$(\vs, \ve, \vli, \vr)$ be an equilibrium point. From the fourth equation in~\eqref{eq:SIRgraph2}, we conclude that~$i_j=0$ for every~$j=1, \dots, n$. By plugging this equality in the third equation of~\eqref{eq:SIRgraph2}, we conclude that~$e_j=0$ for every~$j=1, \dots, n$. This implies that the equilibria are in the form~$\mathbf{P}_\infty =(\vS_\infty ,\, 0 ,\, 0 ,\,\vR_\infty )$. We are left to prove the second part of the lemma. We recall Lemma~\ref{l:bounds} and by using the fourth equation in~\eqref{eq:SIRgraph2} we conclude that, for every~$j=1, \dots, n$, the function~$r_j$ is monotone non-decreasing and hence has a limit as~$t \to + \infty$. Also, the limit is confined between~$0$ and~$1$. By adding the third and the fourth equation in~\eqref{eq:SIRgraph2} we get that ~$i_j + r_j$ is monotone non-decreasing and also has a limit as~$t \to + \infty$. We conclude that~$i_j$ has a limit as~$t \to + \infty$. By using the first equation in~\eqref{eq:SIRgraph2} we infer that, for every~$j=1, \dots, n$,~$s_j$ is monotone non-increasing and hence has a limit, which is confined between~$0$ and~$1$, as~$t\to + \infty$. By adding the first and the second equation in~\eqref{eq:SIRgraph2} we get that, for every~$j=1, \dots, n$, ~$s_j + e_j$ is monotone non-increasing and hence has a limit as~$t \to + \infty$. We eventually conclude that~$e_j$ has a limit for~$t\to + \infty$. Since the limit point is finite, it must be an equilibrium. The condition~$\vs_\infty^{\top} \vone +\vr_\infty^{\top} \vone =\vone$ follows from the equality~$s_j (t) + e_j(t) + i_j (t) + r_j (t) =1$ for every~$j=1, \dots, n$ and~$t \ge 0$.
\end{proof}
We now focus on the transient behaviour of system~\eqref{eq:SIRgraph2}. If~$\vs(t) \ge \vzero$, then the matrix~$B(t)=Diag(\vs(t))\,A$ is a non-negative, irreducible matrix. Owing to the Perron-Frobenius Theorem~\cite{FC1997}, this implies that~$B (t)$ has a positive eigenvalue, which we denote by~$\lambda_M (t)=\rho (B(t))$, equal to the spectral radius. Also, there is an associated positive eigenvector, which we denote by~$\mathbf{V}_M(t)>>0$. Note that, if~$\vs(t) \to 0$, then~$\lambda_M(t) \to 0$. The next result focuses on the transient behaviour of the solutions of~\eqref{eq:SIRgraph2}.
\begin{theorem}\label{Teo:seir2}
Consider system~\eqref{eq:SIRgraph2}, let~$\lambda_M (t)$ be the dominant eigenvalue of~$B(t)$ and~$\mathbf{V}_M(t)>>0$ the corresponding positive left eigenvector.
\begin{itemize}
\item[(1)] Assume that for some~$\tau\geq 0$ we have~$\beta \lambda_M(\tau )\leq \gamma$ then~$q_\tau (t)=\mathbf{V}_M(\tau )^{\top}\, (\ve(t)+\vli(t))$ is monotonically decreasing to~$0$ on~$[\tau, + \infty[$.
\item[(2)] Assume that~$\beta \lambda_M (\tau) > \gamma$ for some~$\tau \ge 0$, then there is~$ t^*>\tau$ such that~$q_\tau(t)= \mathbf{V}_M(\tau )^{\top}\,(\ve (t)+ \vli (t))$ is monotone non decreasing on~$[\tau ,t^*[$.
\item[(3)] If the limit point of the trajectory under consideration is~$(0, 0, 0, 1)$, then there is~$\tau >0$ such that~$\beta \lambda_M (\tau) \leq \gamma$ and
item (1) applies.
\end{itemize}
\end{theorem}
\begin{proof}
We first establish item (1). By multiplying the sum of the second equation and the third equation in~\eqref{eq:SIRgraph2} by~$\mathbf{V}_M(\tau )$ we get
\[
\mathbf{V}_M (\tau)^{\top} \left[ \dot{\ve}(t)+ \dot{\vli}(t)\right] = \vV_M (\tau)^{\top} (\beta\,\Diag(\vs (t)) A \vli(t) - \gamma \vli (t)).
\]
Since~$\vs (t)$ is monotone non increasing (see the proof of Lemma~\ref{Teo:seir1}), then
$$\vV_M (\tau)^{\top} (\beta\,\Diag(\vs (t)) \leq \vV_M (\tau)^{\top} (\beta\,\Diag(\vs(\tau)),$$
for every~$\tau \ge t$. This yields
\[
\frac{d\,\vV_M (\tau)^{\top} (\ve (t) +{\vli}(t))}{dt}\leq \vV_M (\tau)^{\top} (\beta\,\Diag(\vs (\tau )) A \vli (t) - \gamma \vli (t)),\,\,t\geq \tau.
\]
Since~$\vV_M (\tau)^{\top} (\beta\,\Diag(\vs(\tau )) B)=\lambda_M (\tau) \vV_M (\tau)^{\top}$, then
\[
\frac{d\,\vV_M (\tau)^{\top} ({\ve}(t)+\vli (t))}{dt}\leq \left( \beta \lambda_M(\tau ) -\gamma \right)\vV_M (\tau)^{\top} \vli (t),
\leq 0.
\]
To establish the last inequality we have used the inequalities~$\left( \beta \lambda_M(\tau ) -\gamma \right)\leq 0$,~$\vV_M>>0$, and~$\vli (t)\ge 0$. The above inequality implies that~$\vV_M (\tau)^{\top} ({\ve}(t)+\vli (t))$ is monotone non increasing on~$[\tau, + \infty[$, and henceforth converging as~$t \to + \infty$. Owing to Lemma~\ref{Teo:seir1}, the limit is~$0$. To establish item (2), it suffices to point out that, if~$\beta \lambda_M (\tau) > \gamma$, then
$$
\frac{d}{dt}\left( \vV_M(\tau )^{\top} (\ve (t)+ \vli (t)) \right)_{t=\tau}
\ge 0.
$$
To establish item (3), we point out that, if~$\vs (t)$ converges to~$0$, then~$\lambda_M (t)$ converges to~$0$ and hence the condition~$\beta \lambda_M (\tau) \leq \gamma$ is satisfied for sufficiently large~$\tau$.
\end{proof}
\begin{rem}
Theorem~\ref{Teo:seir2} implies that, if the parameters~$\beta$,~$\gamma$, and the matrix~$A$ can change over time, the conditions of the beginning of the epidemic spread or a decrease in quantity~$\ve (t)+ \vli (t)$ could occur several times. This may justify the occurrence of multiple epidemic waves.
\end{rem}
\begin{rem} If the graph~$G$ is a complete graph~$K_n$ ($K_n$ is a simple undirected graph where every pair of distinct vertices is connected by a unique edge), then we recover the scalar SEIR model by setting~$s(t)=\sum_j s_j(t)$,~$e(t)=\sum_j e_j(t)$,~$i(t)=\sum_j i_j(t)$,~$r(t)=\sum_j r_j(t)$.
\end{rem}
\begin{figure}[t]
\centering
\begin{tabular}{cc}
(a)\quad\includegraphics[height=100pt]{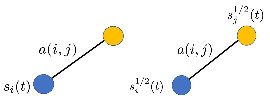}
&(b)\includegraphics[height=100pt]{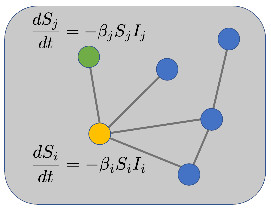}
\end{tabular}
\caption{(a) Un-symmetric and symmetric weights for the SEIR models on a graph, (b) definition of the SEIR model on graphs starting from the governing differential equations of the 1D SEIR model at each node of the input graph.\label{fig:EDGE-WEIGHTS}}
\end{figure}
\section{Spectrum of the SEIR model on graphs\label{sec:spectra}} 
In this section, we always assume that all the components of~$\vs (t)$ are strictly positive. Owing to the proof of Lemma~\ref{l:bounds}, this amounts to assuming that they are strictly positive at~$t=0$. The coefficient matrix \mbox{$B(t):=\Diag(\vs (t))A$} of the SEIR model on a general simple graph, can be interpreted as a weighted adjacency matrix, where the weight associated with the edge \mbox{$(i,j)$} is \mbox{$b_{ij}(t):=s_{i}(t)a_{ij}$}; indeed, the weighted adjacency matrix is not symmetric and the weighted graph is directed. To study the spectrum of \mbox{$B(t)$} and eventually work with a symmetric weighted adjacency matrix, we introduce its ``symmetrisation'' (Fig.,~\ref{fig:EDGE-WEIGHTS})
\begin{equation*}
B_{sym}(t):=\Diag(S(t)) ^{1/2} A\Diag(S(t)) ^{1/2},
\end{equation*}
where the weight associated with the edge \mbox{$(i,j)$} is \mbox{$b^{sym}_{ij}(t):=s_{i}^{1/2}(t)a_{ij}s_{j}^{1/2}(t)$}. Furthermore, the matrices~$ B(t)$ and~$ B_{sym}(t)$ have the same eigenvalues; in fact,
\begin{equation*}
\begin{split}
 B(t)-\nu I
&=\Diag(S(t)) ^{1/2}\left(\Diag(\vs (t)) ^{1/2}\, A\,\Diag( \vs (t)) ^{1/2}-\nu I \right)\Diag( \vs(t)) ^{-1/2}\\
&=Diag(\vs (t)) ^{1/2}\left(B_{sym}(t)-\nu I \right)\Diag( \vs (t)) ^{-1/2}.
\end{split}
\end{equation*}
Let~$ \mathbf{v}~$ be an eigenvector of \mbox{$ B(t)$} with eigenvalue~$\nu$. Then,
\begin{equation*}
\begin{split}
 &B(t) \mathbf{v} =\nu \mathbf{v} \iff
\Diag( \vs (t)) ^{{\color{black}-}1/2} B \mathbf{v} =\nu\Diag(\vs (t)) ^{-1/2} \mathbf{v} \iff
 \mathbf{y} :=\Diag(\vs (t)) ^{-1/2} \mathbf{v} ,\\
 &Diag(\vs (t))^{{\color{black}-}1/2} B\Diag(\vs (t)) ^{1/2} \mathbf{y} =\nu \mathbf{y}
 \iff B_{sym}(t)\mathbf{v}=\nu \mathbf{v}.
 \end{split}
\end{equation*}
Indeed, \mbox{$ \mathbf{y} :=\Diag(\vs (t)) ^{-1/2} \mathbf{v}~$} is an eigenvector of~$ B_{sym}(t)$ with eigenvalue~$\nu$, i.e., the spectrum of~$ B(t)$ is real and it uniquely identifies the spectrum of~$ B_{sym}(t)$, and vice versa. Note that the eigenvectors \mbox{$ \mathbf{y} _{i}:=\Diag(\vs (t)) ^{-1/2} \mathbf{v} _{i}$}, \mbox{$i=1,\ldots,n$} are not orthonormal.

From these properties, we consider the reduced (the vector~$\vr$ can be reconstructed from vectors~$\vs$,~$vli$,~$\ve$) \emph{SEIR} and \emph{symmetric SEIR models on a general simple graph} (Fig.,~\ref{fig:ED-SEIR-TEST})
\begin{equation}\label{eq:SYMMETRIC-SEIR}
\left\{
\begin{array}{l}
\dot{\vs}=-\beta\Diag(\vs (t)) {\color{black}A} \vli
=-\beta B(t) \vli \\
\dot{\ve}=\beta B(t) \vli -\mu \ve \\
\dot{\vli }=\mu \ve -\gamma \vli
\end{array}
\right.\qquad
\left\{
\begin{array}{l}
\dot{\vs }=-\beta B_{sym}(t) \vli \\
\dot{\ve }=\beta B_{sym}(t) \vli -\mu \ve \\
\dot{\vli }=\mu \ve -\gamma \vli .
\end{array}
\right.
\end{equation}
\textbf{Bounds to the spectrum of SEIR on graphs}
Let~$d_{max}$ be the maximum degree of a node in the input graph~${G}$, and let~$d_{avg}$ be the average degree of a node in~${G}$. Let \mbox{$\nu_{1}\geq\nu_{2}\geq\ldots\nu_{n}$} be the eigenvalues of the adjacency matrix~$\tilde{A}$ of~${G}$. Then, the bound \mbox{$d_{avg}\leq\nu_{1}\leq d_{max}$} is applied to estimate the variation of the eigenvalues in linear time. The maximum eigenvalue is computed efficiently through the Arnoldi method~\cite{GOLUB1989}. We now derive upper and lower bounds to the eigenvalues of~$ B(t)$, or equivalently \mbox{$ B_{sym}(t)$} (same eigenvalues):
\begin{equation*}
{\small{\begin{split}
\nu_{1}( B(t))
&=\nu_{1}( B_{sym}(t))
=\max_{ \mathbf{v} \neq 0 }\frac{ \mathbf{v} ^{\top}\Diag(\vs (t)) ^{1/2} {\color{black}A}\Diag( \vs (t)) ^{1/2} \mathbf{v}^{\top}}{ \mathbf{v}^{\top}\mathbf{v} }\\
&=\max_{ \mathbf{y} \neq \vzero }\frac{ \mathbf{y}^{\top} {\color{black}A} \mathbf{y} }{ \mathbf{y}^{\top}\Diag(\vs (t)) ^{-1} \mathbf{y} },\quad \mathbf{y} := {\color{black}Diag(\vs (t)) ^{1/2}} \mathbf{v} ,\\
&\leq\frac{\nu_{1}({\color{black} \tilde A} )+1}{\min_{i}\{s_{i}^{-1}(t)\}}
\leq(\nu_{1}({\color{black} \tilde A)} +1) \max_{i}\{\|s_{i}\|_{\infty}\}
\leq (d_{max}+1)\max_{i}\{\|s_{i}\|_{\infty}\}.
\end{split}}}
\end{equation*}
In an analogous way,
\begin{equation*}
\begin{split}
\nu_{1}( B_{sym}(t))
&={\color{black} \max}_{ \mathbf{v} \neq 0 }\frac{ \mathbf{v}^{\top} B_{symm}(t) \mathbf{v} }{ \mathbf{v}^{\top}\mathbf{}v}
\geq \frac{ \vone ^{\top} B_{sym}(t) \vone }{ \vone ^{\top} \vone }\geq\frac{1}{n}\left( \vone^{\top} B_{sym}(t) \vone \right)\\
&\geq\frac{1}{n}\nu_{1}( A)\sum_{i}s_{i}(t)
\geq( {\color{black} \nu_1(\tilde A)}+1)\min_{i}\{s_{i}(\cdot)\}=(d_{avg}+1)\min_{i}\{s_{i}(\cdot)\}.
\end{split}
\end{equation*}
\textbf{SEIR model and graph Laplacian}
Let~$ D~$ be the diagonal matrix whose entries are the degrees \mbox{$ d :=(d_{i})_{i}$} of the nodes of the input graph and \mbox{$ L := D -\tilde A$} the corresponding Laplacian matrix. Then, we rewrite the SEIR model as
\begin{equation}\label{eq:DIFFUSIVE-SEIR}
\begin{split}
\dot{\vs }
&=-\beta\Diag(\vs (t)) \left[ D - D + {\color{black}A}\right] \vli (t)
=-\beta\Diag(\vs (t)) \left[ D - L +I_n\right] \vli (t)\\
&=-\beta\left[\Diag(\vs (t)) D -\Diag(\vs (t)) L +\Diag(\vs (t))\right] \vli (t)\\
&=-\beta\left[\Diag( (d+1) \circ \vs (t))-\Diag(\vs (t)) L \right] \vli (t),
\end{split}
\end{equation}
where \mbox{$\Diag(\vs (t)) L~$} is the weighted Laplacian matrix.
The continuous counterpart is related to the diffusion with the Laplace-Beltrami operator (Fig.,~\ref{fig:DIFFUSION-SIMULATION}).
\begin{figure}[t]
\centering
\begin{tabular}{cccc}
\includegraphics[height=0.18\textwidth]{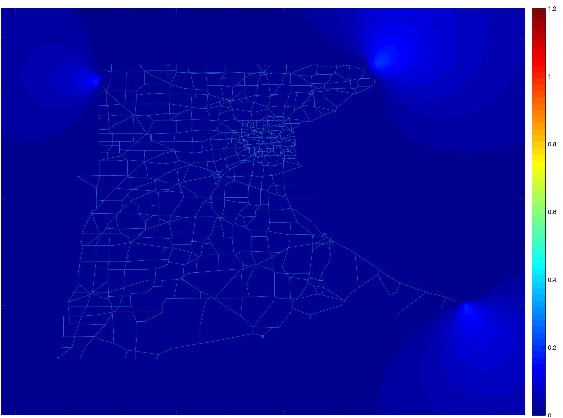}
&\includegraphics[height=0.18\textwidth]{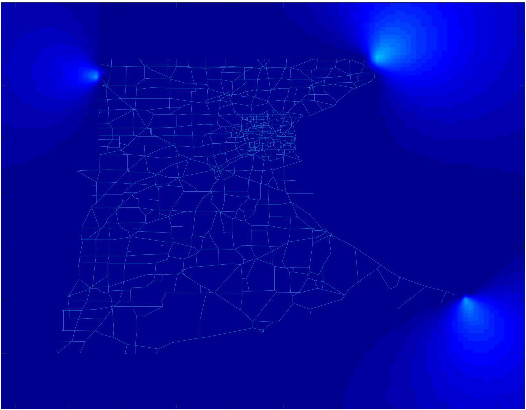}
&\includegraphics[height=0.18\textwidth]{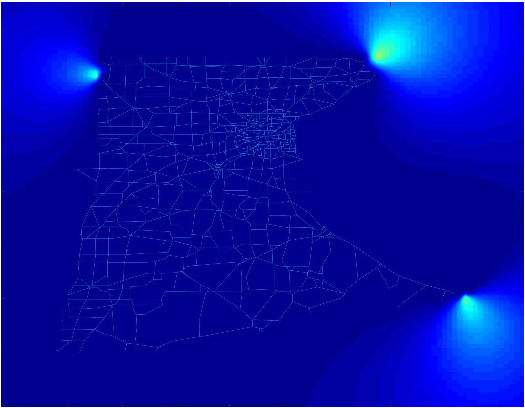}
&\includegraphics[height=0.18\textwidth]{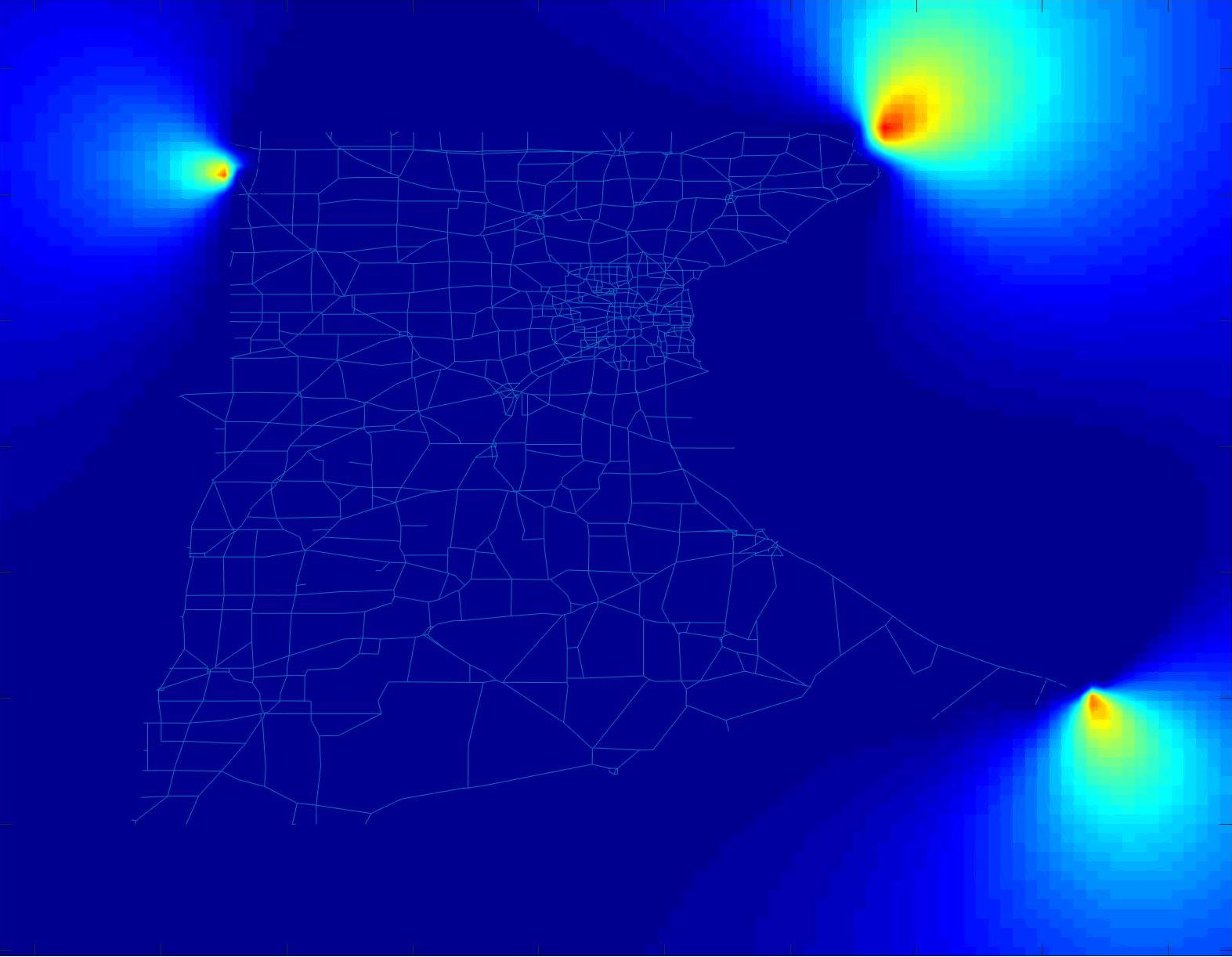}\\
\includegraphics[height=0.18\textwidth]{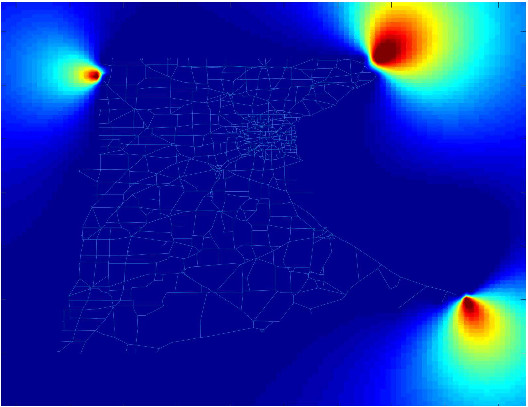}
&\includegraphics[height=0.18\textwidth]{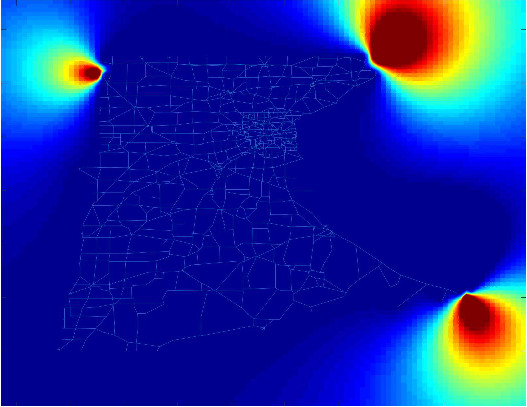}
&\includegraphics[height=0.18\textwidth]{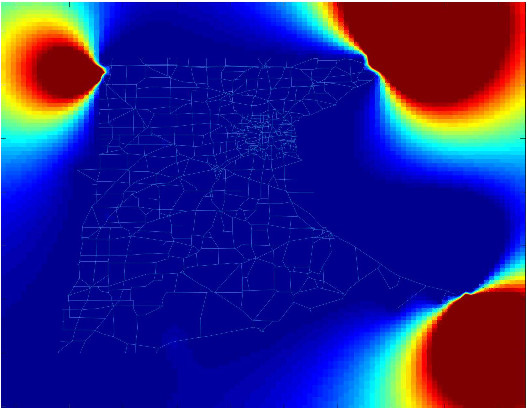}
&\includegraphics[height=0.18\textwidth]{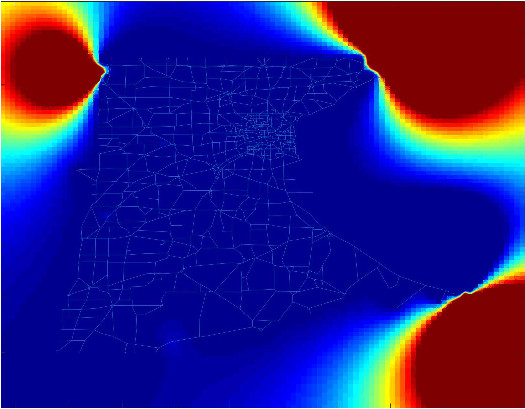}
\end{tabular}
\caption{Colormap of the evolution of the variable \mbox{$I(\cdot)$} (time increases from top to bottom, and from left to right) on a graph with constant weights, according to the symmetric SEIR model in Eq. (\ref{eq:SYMMETRIC-SEIR}).\label{fig:ED-SEIR-TEST}}
\end{figure}
\section{Continuum limit of epidemiological models on graphs\label{sec:graphon}}
We now discuss the continuum limit of compartmental epidemiological models on graph sequences for the case of the SEIR model only, but we are confident our results extend to a large class of compartmental models. Considering the SEIR model defined on a sequence of graphs where the number ~$n$ of vertices diverges, we want to analyse the limit~$n \to + \infty$. To achieve this, we rely on a recently developed theory based on the notion of \emph{graphon}, which is introduced in Sect.~\ref{ss:graphon}. In Sects. ~\ref{ss:convergence},~\ref{sec:GRAPHON-LIMIT}, we discuss the continuous limit of the epidemiological model and the convergence of the SEIR model defined on a converging sequence of graphs. In Sect.~\ref{ss:samplings}, we show that a stronger notion of convergence can be recovered by introducing suitable deterministic and random samplings. Then, we present the semi-discretisation of the SEIR model on a graphon (Sect.~\ref{sec:semid}) and numerical examples (Sect.~\ref{sec:ER}). The analysis in Sects.~\ref{ss:convergence},~\ref{ss:samplings} is based on~\cite{ADS}.
\begin{figure}[t]
\centering
\begin{tabular}{ccc}
\includegraphics[height=110pt]{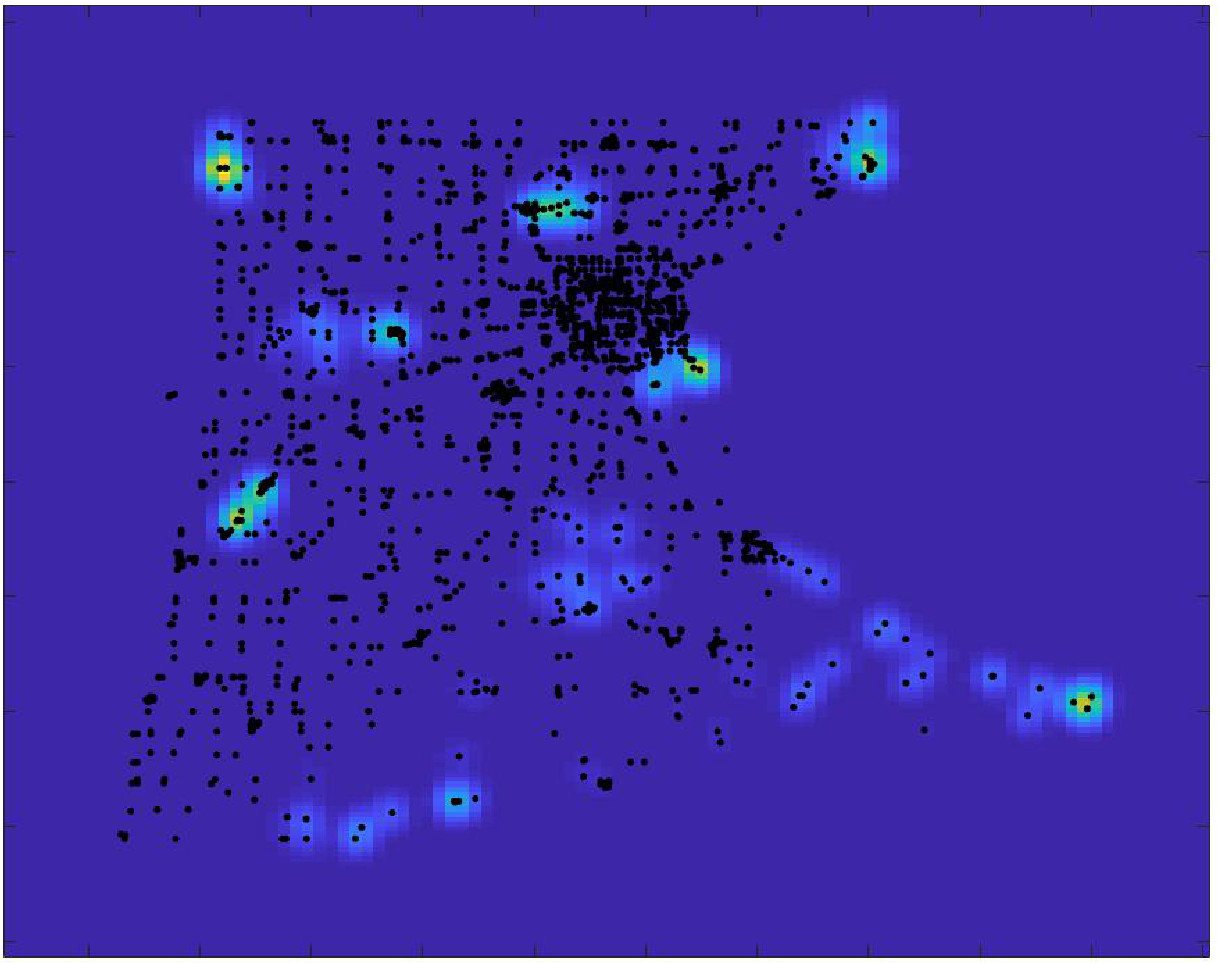}
&\includegraphics[height=110pt]{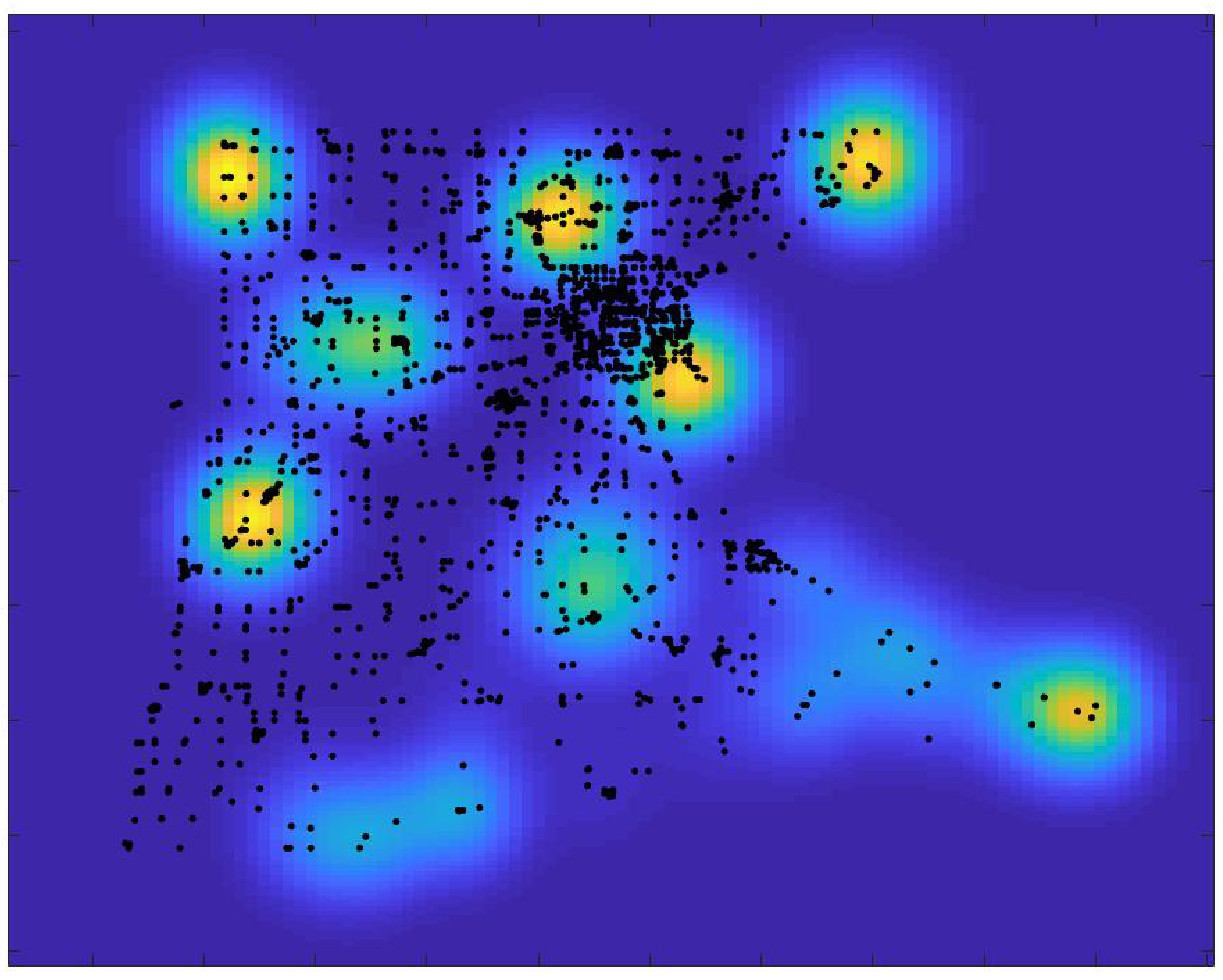}
&\includegraphics[height=110pt]{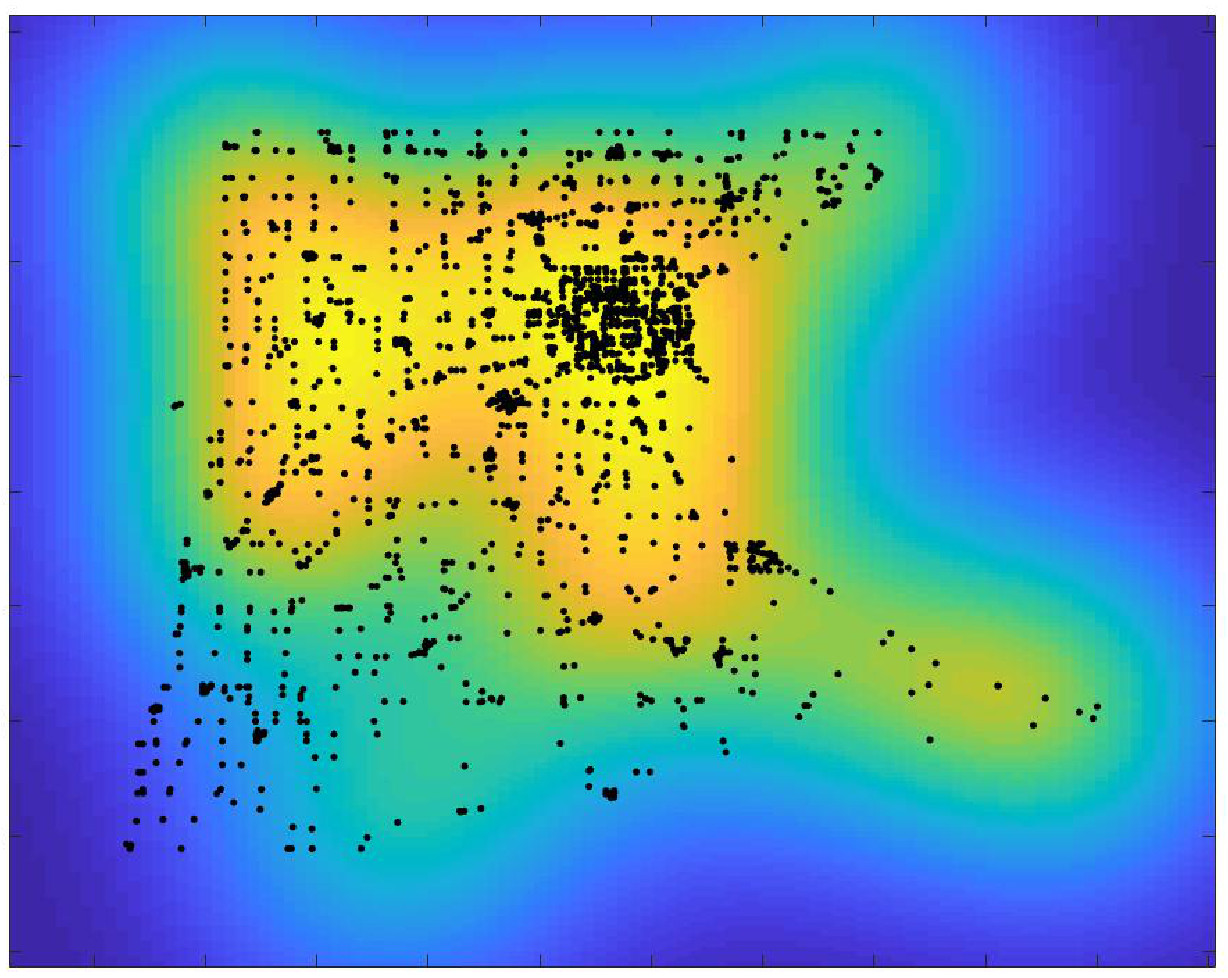}\\
\includegraphics[height=110pt]{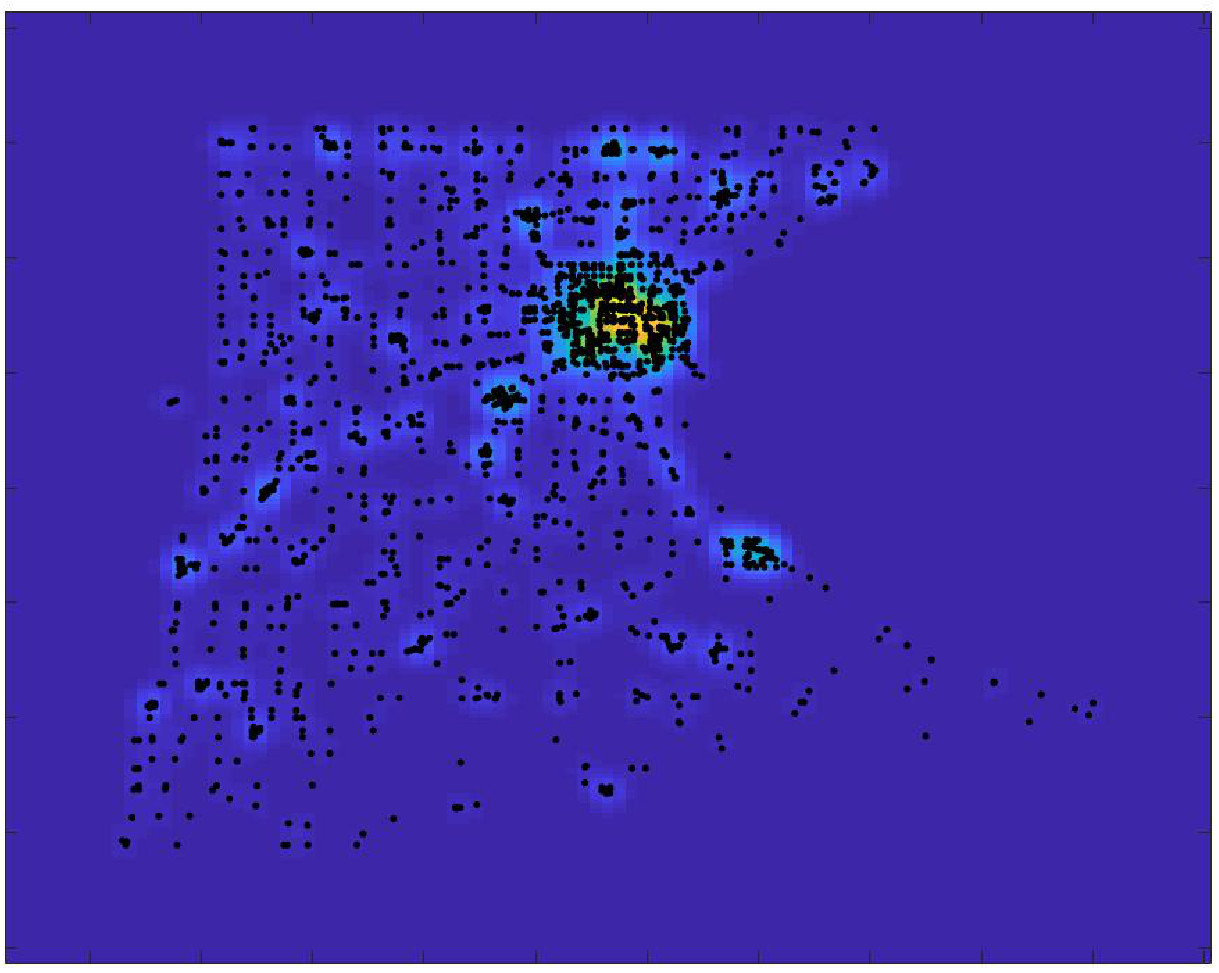}
&\includegraphics[height=110pt]{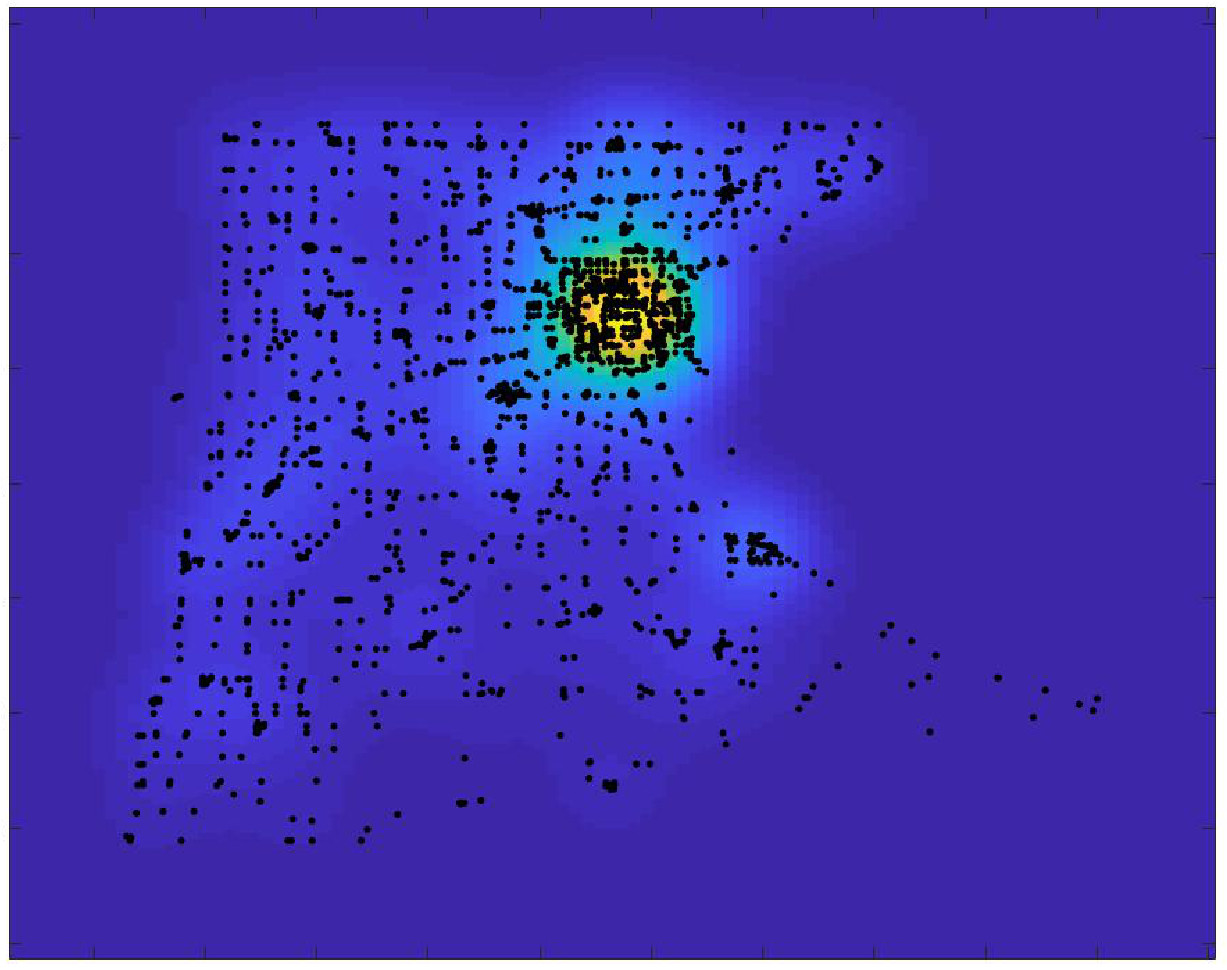}
&\includegraphics[height=110pt]{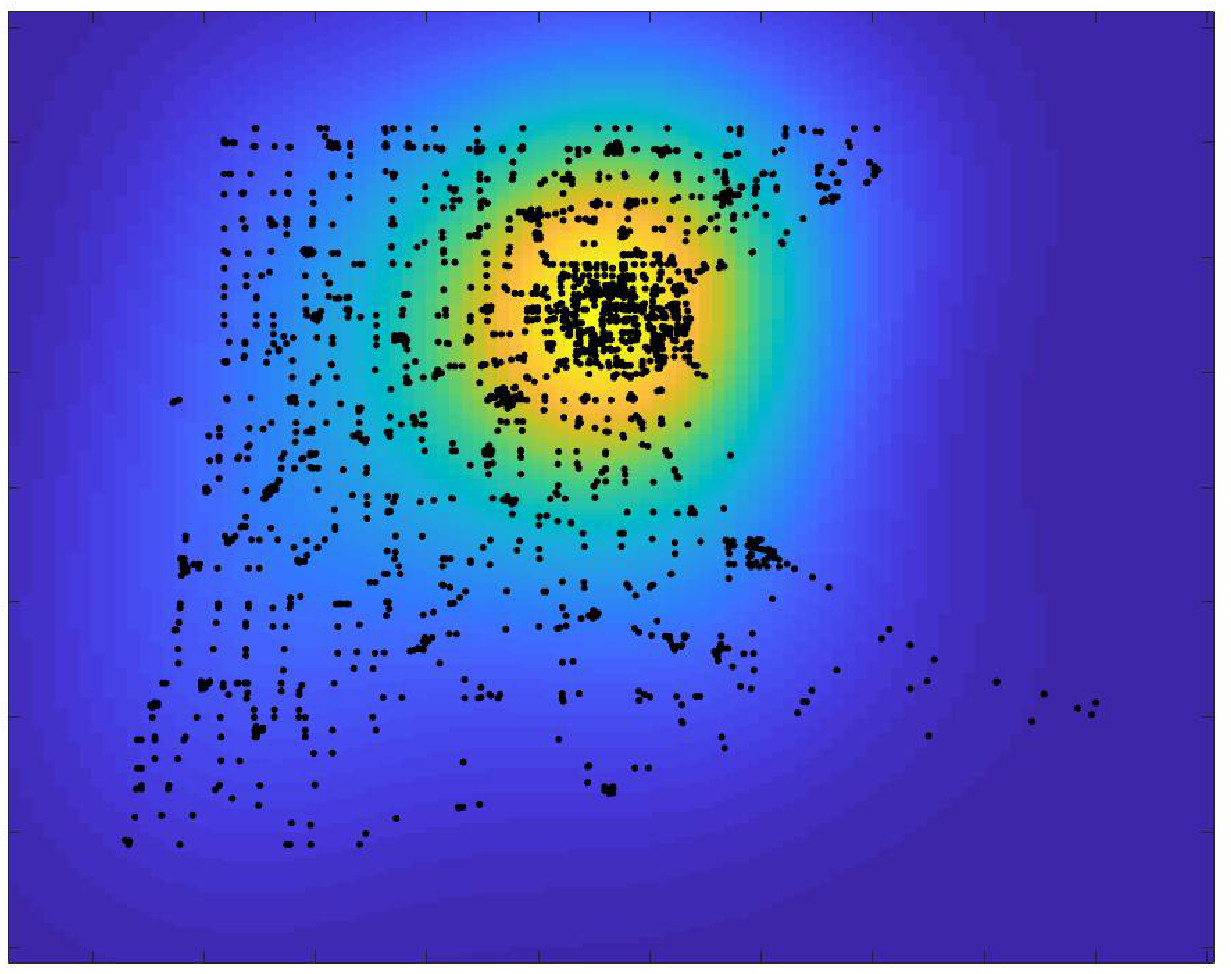}
\end{tabular}
\caption{Colormap of the evolution of the variable \mbox{$I(\cdot)$} (time increases from left to right) on a graph with (1st row) constant weights and (2nd row) positive weights (larger weights in the selected seed point), according to the symmetric SEIR model in Eq. (\ref{eq:DIFFUSIVE-SEIR}).\label{fig:DIFFUSION-SIMULATION}}
\end{figure}
\subsection{Graphons\label{ss:graphon}}
In recent years, a new paradigm for understanding the continuum limit of graph sequences has emerged with the theory of \emph{graphons}~\cite{BCCZ18,BCCZ19,BCLSV06,BCLSV08,BCLSV12,LS07}.
\begin{definition} \label{d:graphon}
A graphon is an integrable function~$W:[0,1]^2\to\rr$ satisfying~$W(x, y) = W(y, x)$ for a.e.~$x,y\in[0,1]$.
\end{definition}
\textbf{The graphon associated with a given graph}
The basic idea underpinning the theory of graphons is to identify the set of nodes of a given graph with the interval~$[0, 1]$. In this way, we can identify the adjacency matrix of an undirected graph with a symmetric function defined on the square~$[0, 1]^2$, that is with a graphon. More precisely, consider a weighted undirected graph~$\G$, with vertices~$\{1, \dots, n\}$. Let~$(\tilde a_{jk})_{j, k=1, \dots, n}$ be its adjacency matrix. Note that we can consider fairly general adjacency matrices as we only assume that the coefficients are nonnegative and symmetric. To construct the graphon associated to~$\G$, consider a partition of~$[0, 1[$ into~$n$ isometric intervals, and term~$I^n_j$ the~$j$-th interval, that is~$I^n_j : = [(j-1)n^{-1}, j n^{-1}[$,~$j=1, \dots,n$. The graphon~$W_{\G}~:~[0, 1[ \times [0, 1[\to \rr$ associated to the graph~$\G$ is defined by setting \mbox{$$$W_{\G}(x, y) = \tilde a_{jk} \; \text{if~$(x, y) \in I^n_j \times I^n_k$}$}. Fig.,~\ref{fig1} illustrates the above notion.\\

\textbf{Converging sequences of graphons}
We define the so-called \emph{cut-norm} of a given graphon by setting
\begin{equation*}
		\|W\|_{\Box}:=\sup_{S,T\subseteq[0,1]}\left\vert\int_{S\times T}W(x,y)\,dxdy\right\vert , \quad
 \text{for every graphon~$W \in L^\infty ([0, 1]^2)$}.
	\end{equation*}
In the above expression, the supremum is taken over all pairs of measurable subsets~$S, T \subseteq [0,1]$. The notion of cut-norm was first introduced by Frieze and Kannan~\cite{FK99}, and its fundamental importance for the theory of graphons was recently unveiled in several papers, see in particular~\cite{BCCZ18,BCCZ19,BCLSV06,BCLSV08,BCLSV12} and the references therein. Here, we limit ourselves to some handwaving remarks. In particular, we point out that the notion of cut-norm for the graphons associated with a sequence of graphs is connected with the notion of \emph{left convergence} for graph sequences, see~\cite[Theorem 3.8]{BCLSV08}. Very loosely speaking, a sequence of graphs is left convergent if the local structure of the graphs is asymptotically stable, in the sense that asymptotically the graphs contain the ``same number of copies" of any given subgraph. For the study epidemics spread, previous work has addressed several aspects such as spectral representations of graphons~\cite{SHUANG2019}, sensitivity analysis of SIS epidemics~\cite{VIZUETE2020}, and Finite state graphon games~\cite{AURELI2022}.

\subsection{The continuum limit of epidemiological models\label{ss:convergence}}
We now discuss our main result concerning the continuum limit of the SEIR model defined on sequences of graphs. For the system \eqref{eq:SIRgraph2}, we now introduce the possibility that the parameters~$\beta$,~$\gamma$,~$\mu$, are positive time-dependent functions. In this way, it is possible to introduce in the model any changes over time due to external interventions (e.g. lockdown or social distancing) rather than the evolution of the biomedical parameters related to the epidemic. We first have to introduce some notation. Fix a graph~$\G_n$ with~$n$ vertices~$\{1, \dots, n\}$ and consider the following SEIR model defined on the set of vertices:
\begin{equation}\label{SIR_G}
		\begin{cases}
		\frac{ds_j^n(t)}{dt}=-s_j^n(t)\f1{n}\sum_{k=1}^{n}\beta_k^n(t)a_{jk}^n
		i_k^n(t)& \\
		\frac{de_j^n(t)}{dt}=s_j^n(t)\f1{n}\sum_{k=1}^{n}\beta_k^n(t)a_{jk}^n i_k^n(t)-\mu_j^n(t) e_j^n(t)
		 &\qquad j=1,\,\dots,\,n\\
		\frac{di_j^n(t)}{dt}= \mu_j^n(t) e_j^n(t)-\gamma_j^n(t) i_j^n(t) & \\
 		\frac{dr_j^n(t)}{dt}= \gamma_j^n(t) i_j^n(t) & \\
		s_j^n(0)=s_{j,0}^n,\,e_j^n(0)=e_{j,0}^n,\, i_j^n(0)=i_{j,0}^n,\,r_j^n(0)=r_{j,0}^n .
		\end{cases}
	\end{equation}
Also,~$(a^n_{jk})_{j, k=1, \dots, n}$ denotes the weighted adjacency matrix of the graph~$\G_n$ with~$a^n_{jj}=1$. Note that~\eqref{SIR_G} boils down to~\eqref{eq:SIRgraph2} provided the coefficients do not depend on time and
$$
 a_{jk} =
 \left\{
 \begin{array}{ll}
    \tilde a_{jk} & j \neq k \\
    1 & j=k.
 \end{array}
 \right.
$$
In the following, we use a more compact notation and denote by~$\svet^n, \evet^n,\ivet^n,\rvet^n:[0,+\infty)\to[0,1]^n$ the vectors with~$j$-th component given by~$s_j^n, e_j^n, i^n_j, r^n_j$,~$j=1, \dots, n$, respectively. We will assume the functions~$\beta_k^n(t)$,~$\gamma_j^n(t)$,~$\mu_j^n(t)$ are continuous. More general hypotheses can however be considered considering also discontinuous, but bounded, as for the theory of the Carath\'{e}odory differential equations, see e.g.~\cite{FIL88}: the following results also extend to this case. We now aim at discussing the limit~$n \to + \infty$ of~\eqref{SIR_G} by applying the theory of graphons. To achieve this, we, first of all, have to regard the functions~$\svet^n,\ivet^n,\rvet^n: \rr^+ \to[0,1]^n$ as functions defined on~$\rr^+ \times [0, 1]$. To this end, in the following, we always identify any function~$\texttt{u}: \rr^+ \to\rr^n$,~$\texttt{u}=\left(u_j\right)_{j\in V(\G_n)}$, with the piecewise constant function
\begin{equation}\label{uG->u01}
\texttt{u}_n(t, x):=\sum_{j=1}^{n}u_j (t) \mathbbm{1}_{I_j^n}(x)\,
\end{equation}
where~$\mathbbm{1}_{I_j^n}$ is the characteristic function of the interval~$I_j^n=[ (j-1)n^{-1}, jn^{-1}[$. Furthermore, for a function~$v(t,x)$,~$(t,x)\in (\mathbb{R}^+\times [0,1])$, we set~$v_j^n(t):=\langle v (t, \cdot) \rangle_{I^n_j }$, where we denote with~$\langle v \cdot \rangle_{I^n_j }$ the average of the function on the set~$I^n_j : =[(j-1)n^{-1}, jn^{-1}[$, and~$v_n(t,x)=\sum_{j=1}^{n} v_j^n (t) \mathbbm{1}_{I_j^n}(x)$. Whenever we say that a sequence of functions defined on the vertices converges to a limit function defined on~$[0,1]$, we always mean that the corresponding piecewise constant functions are defined as in \eqref{uG->u01} converge to the given limit. We are now ready to introduce the discrete graphon representation of~\eqref{SIR_G}
\begin{equation}
		\label{SEIR_N}
		\begin{cases}
		\partial_t s_n(t,x)=-s_n(t,x) \int_0^1\beta_n(t,y)W_{\G_n}(x,y)i_n(t,y)\,dy, & \\
		\partial_t e_n(t,x)=s_n(t,x),
 		{\int_0^1\beta_n(t,y)W_{\G_n}(x,y)i_n(t,y)\,dy}-\mu_n(t,x)e_n(t,x),
 		&x\in[0,1],\\
		\partial_ti_n (t,x)= \mu_n(t,x)e(t,x)-\gamma_n(t,x)i(t,x), & \\
 		\partial_tr_n(t,x)=\gamma_n(t,x)i_n(t,x), & \\
		s_n(0,x)=s_0(x),\,e_n(0,x)=e_0(x),\,i_n(0,x)=i_0(x),\,r_n(0,x)=r_0(x),
		\end{cases}
	\end{equation}	
where~$s_0(x),\,e_0(x),\,i_0(x),\,r_0(x)$ are suitable piecewise constant functions. The dynamics of the discrete graphon model can be studied using the original ODE system.
\begin{lemma}\label{Lem:disgraphon}
Let~$n>1$ and~$(s_j^n(t),\,e_j^n(t),\,i_j^n(t),\, r_j^n(t)),\, t\geq 0,\,j=1,...,n$ the solution of the Cauchy problem \eqref{SIR_G}, then the corresponding piecewise constant functions~$(s_n(t,x),\,e_n(t,x),\, i_n(t,x),\, r_n(t,x))$, is the solution, formulated almost everywhere, of the Cauchy problem \eqref{SEIR_N},
$t\geq 0,\,x\in [0,1]$.
\end{lemma}
\begin{proof}
First, we observe that the map~$t \mapsto (s_n (t),\,e_n(t),\,i_n(t),\, r_n(t))$ is differentiable for almost everywhere~$x\in [0,1]$ and~$t>0$. Then, combining for example the~$s_j^n$ components from \eqref{SIR_G},
\begin{equation*}\label{eq:superpos}
\begin{split}
\partial_t s_n(t,x)
&= - \sum_{j=1}^{n} \left( s_j^n(t)\f1{n}\sum_{k=1}^{n}\beta_k^n(t)a_{jk}^n i_k^n(t) \right) \mathbbm{1}_{I_j^n}(x) \\
&= - \sum_{j=1}^{n} s_j^n(t) \mathbbm{1}_{I_j^n}(x) \left( \int_{(k-1)/n}^{k/n} \beta_k^n(t)a_{jk}^n i_k^n(t) \right) \\
&= - \sum_{j=1}^{n} s_j^n(t) \mathbbm{1}_{I_j^n}(x) \left( \int_0^1\beta_n(t,y)W_{\G_n}(x,y)i_n(t,y)\,dy \right) \\
&= -s_n(t,x) \int_0^1\beta_n(t,y)W_{\G_n}(x,y)i_n(t,y)\,dy.
\end{split}
\end{equation*}
The same can be done for the other components of the system \eqref{SEIR_N}. Finally, for the initial condition just use the definition of the functions~$s_n(0,x),\,e_n(0,x),\,i_n(0,x),\,r_n(0,x)$.
\end{proof}
If we consider system \eqref{SIR_G} or, equivalently, the system \eqref{SEIR_N} as a quadrature rule which uses a piecewise constant approximation, we can formally introduce the limit ``continuous counterpart'' of~\eqref{SIR_G}, or \eqref{SEIR_N}:
\begin{equation}
		\label{SIR_W}
		\begin{cases}
		\partial_t s(t,x)=-s(t,x) \int_0^1\beta(t,y)W(x,y)i(t,y)\,dy & \\
		\partial_t e(t,x)=s(t,x)
	 	\int_0^1\beta(t,y)W(x,y)i(t,y)\,dy-\mu(t,x)e(t,x)
 		& \qquad x\in[0,1]\\
		\partial_ti(t,x)= \mu(t,x)e(t,x)-\gamma(t,x)i(t,x) & \\
 		\partial_tr(t,x)=\gamma(t,x)i(t,x) & \\
		s(0,x)=s_0(x),\,e(0,x)=e_0(x),\,i(0,x)=i_0(x),\,r(0,x)=r_0(x)
		\end{cases}
	\end{equation}
which is the graphon SEIR model (G-SEIR model).
\begin{figure}[t]
\centering
\begin{tabular}{c}
\includegraphics[height=180pt]{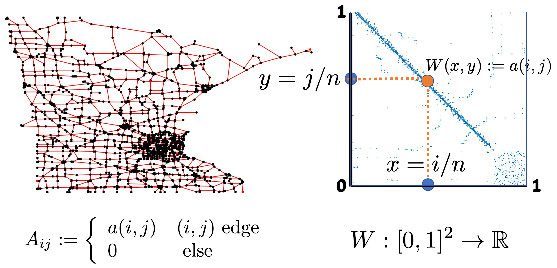}
\end{tabular}
\caption{Generation of the graphon starting from the adjacency matrix of the input graph.\label{fig1}}
\end{figure}
\subsection{Well-posedness and properties of the graphon limit\label{sec:GRAPHON-LIMIT}}
The graphon form \eqref{SIR_W} of the SEIR model is an infinite-dimensional bilinear dynamical system, where each point~$x\in [0,1]$ corresponds to a (sub-)community. While the dynamics of \eqref{SEIR_N} describes the evolution of a piecewise constant macroscopic approximation. In the following, we will make the following assumptions on the data of the graphon model \eqref{SIR_W}.
\begin{assumptionp}{G-SEIR model}\label{Ass:GSEIR} {\it Assume that the following hypothesis holds.
\begin{itemize}
\item[i)]~$W\in L^\infty ([0,1]^2)$ is a non negative graphon~$W\geq 0$ a.e. on~$[0,1]^2$, then~$W\in L^p ([0,1]^2),\, p\geq 1$. There is a constant~$K_w>0$ such that
$ess\sup |W |\leq K_w$.
\item[ii)]~$\beta,\,\gamma , \, \mu \in L^\infty (\mathbb{R}_+ \times [0,1])$, there are constants~$c_0>0,\,K_0>$ such that~$\inf \beta,\, \inf \gamma ,$
$\inf \mu \geq c_0$, and~$ess\sup \vert\beta\vert,\,ess\sup \vert\gamma\vert,\,ess\sup \vert\mu\vert\leq K_0$.
\item[iii)]~$s_{0},e_0, i_{0},r_{0} \in L^\infty([0,1])$, and ~$0 \leq s_0, e_0, i_0, r_0 \leq 1$, ~$s_{0}(x)+e_0(x) +i_{0}(x)+r_{0}(x)=1$ a.e. on~$[0,1]$.
\end{itemize}
}
\end{assumptionp}
Furthermore, we denote by~$B$ the Banach space~$B_{\infty}=(L^\infty [0,1])^4$, with norm~$\| \textbf{u}\|_{B_\infty}~=~\sum_{i=1}^{4} \| u_i \|_{L^\infty}$, where
$\textbf{u}=(u_1,\,u_2,\,u_3,\, u_4)\in B_{\infty}$ and~$\| v \|_{L^\infty} = \| v \|_{L^\infty ([0,1])} = ess\sup_{x\in [0,1]}~|v(x)|$; the positive cone in~$B_\infty$
\[
B_\infty^+ = \left\{ \textbf{u}=(u_1,\,u_2,\,u_3,\, u_4)\in B_{\infty}\,:\, u_i(x)\geq 0\,\,a.e.\, on \,[0,1] \right\},
\]
is a proper cone in~$B_\infty$ with non-empty interior; finally the space~$C^0([0,T];B_\infty )$ is the space of continuous function~$v:[0,T] \rightarrow B_\infty$ with norm
\begin{equation*}\label{eq:norm:C0B}
\| v \|_{C^0([0,T];B_\infty )} ~=~ \sup_{0\leq t \leq T}~\| v(t) \|_{B_\infty}.
\end{equation*}
Then the system \eqref{SIR_W} can be studied as an ordinary differential equation in Banach space~$B_\infty$. To demonstrate the well-posedness of the problem \eqref{SIR_W} we first consider some a-priori estimates for the solution and then we prove the existence and local uniqueness of the solution, and then extend it to a global solution for every~$t>0$. While it is straightforward to verify an upper bound for the solution, it is necessary to establish the non-negativity of the solution itself. In particular, we will prove the invariance of the cone~$B_\infty^+$ with respect to the Cauchy problem \eqref{SIR_W}.

Let~$X$ a Banach space and~$T>0$, we consider a function~$G: [0,T) \times X \rightarrow X$ and we we suppose that the solution~$u:[0,T) \rightarrow X$ of the following Cauchy problem
\begin{equation}\label{eq:general:ode}
u^\prime (t) = G(t,u(t)),~t\in [0,T), \qquad
u(0)=u_0\in X,
\end{equation}
exists and is unique. We recall that a set~$A\subset X$ is said to be forward invariant with respect to~$G$ if the solution~$u(t),~t\in [0,T)$ of \eqref{eq:general:ode} takes values in~$A$ provided that~$u_0\in A$ (see e.g.~\cite{ODEBanach} section 5.2). Theorems on flow-invariance are well-known in the case~$X=\mathbb{R}^n$, for the general case we refer the reader to~\cite{RedWal1975,RedWal1986,Brezis1970}. In most applications, the set~$A$ has a structure which makes the forward invariance easier to show. Suppose that the positive cone~$X_+$ is a non-empty closed subset of~$X$,~$X^*_+$ the space of all continuous linear functionals on~$X_+$, from the Theorem 1. and the property {\it (f)} of section 4 in~\cite{RedWal1975}, we can deduce the following result.
\begin{lemma}\label{lem:nonneg}
Let~$X_+$ a proper cone of~$X$ with non-empty interior, if for all~$(t,y) \in [o,T) \times \partial X_+$, and for all~$f\in X_+^*$, such that~$f(y)=0$ we have~$f(G(t,y))\geq 0$, then~$X_+$ is a forward invariant set for the Cauchy problem \eqref{eq:general:ode}.
\end{lemma}
Let~$C^1(\mathbb{R}_+;B_\infty)$ the Banach space of differentiable function~$v:\mathbb{R}_+ \rightarrow B_\infty$, we state our main result about the Cauchy problem
\eqref{SIR_W}.
\begin{theorem}\label{teo:SEIR:W}
Let's assume that hypotheses {\bf (\ref{Ass:GSEIR})} hold. Then there exists a unique solution~$(s,~e,~i,~r)\in C^1(\mathbb{R}_+;B_\infty)$ of the Cauchy problem \eqref{SIR_W}.
\end{theorem}
\begin{proof}
\textit{Local existence and uniqueness.} The system \eqref{SIR_W} is in the form of \eqref{eq:general:ode} where~$X=B_\infty$ and for~$\textbf{u}=(s,~e,~i,~r)$
{\small{
\begin{equation*}\label{eq:def:Gu}
G(t,\textbf{u})=
 \left[
\begin{array}{c}
G_1(t,\textbf{u}) \\
G_2(t,\textbf{u}) \\
G_3(t,\textbf{u}) \\
G_4(t,\textbf{u})
\end{array}
\right]
~=~
 \left[
\begin{array}{l}
 -s(t,x) \int_0^1\beta(t,y)W(x,y)i(t,y)\,dy \\
 s(t,x) \int_0^1\beta(t,y)W(x,y)i(t,y)\,dy-\mu(t,x)e(t,x) \\
 \mu(t,x)e(t,x)-\gamma(t,x)i(t,x) \\
 \gamma(t,x)i(t,x)
\end{array}
\right].
\end{equation*}}}
Let~$\mathbf{u}=((s,~e,~i,~r),\,\mathbf{\hat{u}}=((\hat{s},~\hat{e},~\hat{i},~\hat{r})\in B_\infty$ such that~$\|\mathbf{u} \|_{B_\infty},\,\| \mathbf{\hat{u}} \|_{B_\infty} \leq \bar{C}$ for a suitable constant~$\bar{C}>0$. Then,
\begin{equation*}
\begin{split}
&\vert G_1(t,\mathbf{u}) -G_1 (t,\mathbf{\hat{u}})\vert\\
&= \left\vert -s(t,x) \int_0^1\beta(t,y)W(x,y)i(t,y)\,dy + \hat{s}(t,x)\int_0^1\beta(t,y)W(x,y)\hat{i}(t,y)\,dy \right\vert \\
&\leq \left\vert -s(t,x) \int_0^1\beta(t,y)W(x,y)i(t,y)\,dy + \hat{s}(t,x) \int_0^1\beta(t,y)W(x,y){i}(t,y)\,dy \right\vert \\
&+ \left\vert -\hat{s}(t,x) \int_0^1\beta(t,y)W(x,y)i(t,y)\,dy + \hat{s}(t,x) \int_0^1\beta(t,y)W(x,y){i}(t,y)\,dy \right\vert \\
&\leq K_w \bar{C} K_0 ( \| s - \hat{s}\|_{L^\infty} + \| i - \hat{i}\|_{L^\infty}).
\end{split}
\end{equation*}
Similarly for the other components of~$G$,
\[
\begin{array}{l}
 \vert G_2(t,\mathbf{u}) -G_2 (t,\mathbf{\hat{u}})\vert \leq K_w \bar{C} K_0 ( \| s - \hat{s}\|_{L^\infty} + \| i - \hat{i}\|_{L^\infty}) + K_0 \| e - \hat{e}\|_{L^\infty}, \\
 \vert G_3(t,\mathbf{u}) -G_3 (t,\mathbf{\hat{u}})\vert \leq K_0 ( \| e - \hat{e}\|_{L^\infty} + \| i - \hat{i}\|_{L^\infty})\\
 \vert G_4(t,\mathbf{u}) -G_4 (t,\mathbf{\hat{u}})\vert \leq K_0 \| i - \hat{i}\|_{L^\infty}
\end{array}
\]
therefore the following estimate holds uniformly with respect to time,
\begin{equation*}\label{eq:lip:G}
\| G(t,\mathbf{u}) -G (t,\mathbf{\hat{u}})\|_{B_\infty} \leq C^*~\| \mathbf{u} - \mathbf{\hat{u}}\|_{B_\infty},
\end{equation*}
where~$C^*=\max \{ K_w\bar{C}K_0,~K_0\}$. Then~$G$ is a local Lipschitz function and we have the existence and uniqueness of a local solution~$\mathbf{u}(t)$, for~$t\in [0,\tau ]$, with~$\tau >0$ for the Cauchy problem \eqref{SIR_W} (see e.g. Section 1 of~\cite{ODEBanach}). Moreover the solution~$\mathbf{u}\in C^1([0,\tau ]; B_\infty)$.

\textit{Global solution,~$t\in\mathbb{R}_+$.} Adding the equations of the system \eqref{SIR_W} we obtain
\[
\partial_t (s(t,x)+e(t,x)+i(t,x)+r(t,x))=0\,\, a.e.\,\,on\,\, [0,1].
\]
From the assumptions on the initial data, it can be concluded that
\[
s(t,x)+e(t,x)+i(t,x)+r(t,x)=1\,\, a.e.\,\,on\,\, [0,1].
\]
Applying Lemma~\ref{lem:nonneg} with~$X=B_\infty$ and~$X^+=B_\infty^+$ it is possible to show that the components of the solution are almost everywhere non-negative. First of all, a null component~$s,~e,~i,~r$ belongs to the boundary~$\partial B_\infty^+$ of~$B_\infty^+$, and~$G(t,\mathbf{u})$ can be rewritten as a sum of vectors with all but one null component. For a single function~$v\in L^\infty_+$, the positive cone of the space~$L^\infty$, and a continuous linear functional~$f$ on~$L^\infty_+$ such that~$f(v)=0$, then for all~$z\in L^\infty$, we have~$f(z~v)=0$. Since~$v$ is a.e. non-negative, we have
\[
-\| z \|_{L^\infty}~v~\leq z~v \leq ~\| z \|_{L^\infty}~v,\qquad
-\| z \|_{L^\infty}~f(v)~\leq f(z~v) \leq ~\| z \|_{L^\infty}~f(v),
\]
but~$f(v)=0$, then~$f(z~v)=0$. For example, for the component~$G_1$, and other equal to zero, if~$s \geq 0$ a.e.,~$-(\int_0^1\beta(t,y)W(x,y)i(t,y)\,dy)\in L^\infty [0,1]$, and the value of the functional is equal to zero. The same happens for all the terms of the decomposition of~$G (t,\mathbf{u})$, then the hypotheses of the Lemma~\ref{lem:nonneg} are verified and the positive cone~$B_\infty^+$ is a set forward invariant for the Cauchy problem \eqref{SIR_W}. As a consequence~$0 \leq s(t,x),\, e(t,x),\, i(t,x),\, r(t,x)$ almost everywhere, so the solution is uniformly bounded and we can define it globally for each~$t\in\mathbb{R}_+$.
\end{proof}
\begin{rem}
We point out that it is not possible to apply techniques based on monotone~\cite{Show1997} or quasi-monotone operators~\cite{Volk1972} because of the sign in the term of the first two components of~$G$. In the case of a SIR model, without the Exposed sub-population, if we do not consider the Removed class it is possible to rewrite the final system with a quasi-monotone operator. We also observe that Theorem~\ref{teo:SEIR:W} can be applied also to the discrete graphon model \eqref{SEIR_N} with minor changes.
\end{rem}
We now consider the convergence of the solutions of the discrete graphon model \eqref{SEIR_N} (or equivalently of the solutions of the discrete model \eqref{SIR_G}) to the solutions of the continuous graphon model \eqref{SIR_W}. Since we consider non-negative graphs we will adopt for convergence instead of the cut-norm \eqref{cut1} the~$L^1$ norm. For bounded graphon~$W$ the following inequalities hold (see~\cite{LO12}, section 12.4),
\[
\|W\|_{\Box} \leq \| W \|_{L^1}\leq \| W \|_{L^2} \leq \| W \|_{L^1}^{1/2}
\]
and convergence with respect to the~$L^1$ norm implies convergence with respect to the cut-norm. Moreover for a~$0-1$ valued graphon~$W$ which is close to our graphon for our epidemiological application, and for a sequence~$\{W_n\}$ of graphons such that~$\|W_n -W \|_{\Box} \rightarrow 0$, then~$\|W_n -W \|_{L^1} \rightarrow 0$ (\cite{LO12}, Proposition 8.24).
\begin{theorem}\label{t:grafoni1} Let~$T>0$, and~$(\G_n)_{n}$ be a sequence of undirected graphs with vertices~$\{1, \dots, n\}$ such that the entries~$a_{jk}^n$ of the~$n$-th weighted adjacency matrix satisfy the following assumptions: i)~$a_{jk}^n=a_{kj}^n$ for every~$j, k=1, \dots, n$; ii)~$a_{jk}^n\geq0$, for every~$j, k=1, \dots, n$; iii) there is~$K_G>0$ such that
 \begin{equation}\label{bound_deg_G}
 \sup_{n\in\nn} \sup_{j=1, \dots, n}\f{1}{n}\sum_{k=1}^na_{jk}^n \leq K_G.
 \end{equation}	
Also, assume that there is a nonnegative graphon~$W\in L^\infty \left( [0,1]^2\right)$ such that
\begin{equation}\label{e:conv_tloc1}
 \left\| W_{\G_n}-W\right\|_{L^1} \to 0\qquad\text{ as }n\to+\infty\,.
 \end{equation}
 Assume furthermore that~$s_{j, 0}^n,e_{j,0}^n,i_{j,0}^n,r_{j,0}^n\geq0$ satisfy~$s_{j,0}^n+e_{j,0}^n+i_{j,0}^n+r_{j,0}^n=1$, that
 ~$$
 \| \beta^n_j \|_{L^\infty(\rr^+)}, \| \mu^n_j \|_{L^\infty(\rr^+)}, \| \gamma^n_j \|_{L^\infty(\rr^+)} \leq M, \; \text{for every~$j=1, \dots, n$ and~$n \in \mathbb N$}
 ~$$
 for a suitable constant~$M>0$ and that
 ~$$
 \beta^n \to \beta, \; \gamma^n \to \gamma, \; \mu^n \to \mu \;
 \text{strongly in~$L^1 (\mathbb R_+ \times [0, 1])$},
 ~$$
 for some~$\beta, \gamma, \mu \in L^\infty (\mathbb R_+ \times [0, 1])~$ such that~$\inf \beta ,\,\inf \gamma , \, \inf \mu >0$.
Let~$\mathbf{u}_n = (s_n,~e_n,~i_n,~r_n)$,~$n=1,2,\ldots$, the solution of the discrete graphon Cauchy problem \eqref{SEIR_N} for~$t\in [0.T]$, with initial data
$s_n(0,x),~e_n(0,x),~i_n(0,x),~r_n(0,x) \in L^\infty([0,1])$, such that~$0 \leq s_n(0,x), e_n(0,x), i_n(0,x), r_n(0,x) \leq 1$, ~$s_n(0,x)+e_n(0,x) +i_n(0,x)+r_n(0,x)=1$ a.e. on~$[0,1]$. Suppose that
\[
 s_n(0,x) \to s_0(x) , \, e_n(0,x) \to e_0(x) , \, i_n(0,x) \to i_0(x) , \, r_n(0,x) \to r_0(x) , \,
 \text{in~$L^\infty ([0, 1])$},
\]
with~$0 \leq s_0, e_0, i_0, r_0 \leq 1$, ~$s_0(x)+e_0(x) +i_0(x)+r_0(x)=1$ a.e. on~$[0,1]$. Then the sequence~$\{ \mathbf{u}_n \}$ converges to the solution~$\mathbf{u}$ of the continuous graphon Cauchy problem \eqref{SIR_W},~$t\in [0.T]$, in the norm of the space~$C^0([0,T];B_\infty )$.
\end{theorem}
\begin{proof}
The assumptions {\bf (\ref{Ass:GSEIR})} are verified. The Cauchy problems \eqref{SIR_W} and \eqref{SEIR_N} can be rewritten equivalently as integral equations (mild solution),
\begin{equation*}\label{eq:mild:solution}
\mathbf{u}(t,x)= \mathbf{u}(0,x)+\int_{0}^{t} G(\tau ,\mathbf{u}(\tau ,x)) d\tau,\qquad
\mathbf{u}_n (t,x)= \mathbf{u}_n (0,x)+\int_{0}^{t} G_n(\tau,\mathbf{u}_n(\tau,x)) d\tau ,
\end{equation*}
where
\begin{equation*}\label{eq:def:Gn}
G_n(t,\textbf{u}_n)=
 \left[
\begin{array}{c}
G_{n1}(t,\textbf{u}) \\
G_{n2}(t,\textbf{u}) \\
G_{n3}(t,\textbf{u}) \\
G_{n4}(t,\textbf{u})
\end{array}
\right]
~=~
 \left[
\begin{array}{l}
 -s_n(t,x) \int_0^1\beta_n(t,y)W_{\G_n}(x,y)i_n(t,y)\,dy \\
 s_n(t,x) \int_0^1\beta_n(t,y)W_{\G_n}(x,y)i_n(t,y)\,dy-\mu_n(t,x)e_n(t,x)\\
 \mu_n(t,x)e_n(t,x)-\gamma_n(t,x)i_n(t,x) \\
 \gamma_n(t,x)i_n(t,x)
\end{array}
\right].
\end{equation*}
In the following, we will omit explicitly writing the variables~$x$ and~$y$ to simplify the notation when the context is clear.
Subtracting component by component~$\mathbf{u}$ from~$\mathbf{u}_n$ as mild solutions we have,
\begin{equation*}\label{eq:u:meno:un}
{\small{
\begin{array}{l}
s(t)-s_n(t) = (s(0)-s_n(0))+{\int_{0}^{t} \left[s_n(\tau) \int_0^1\beta_n(\tau)W_{\G_n} i_n(\tau )\,dy-
 s(\tau) \int_0^1\beta (\tau)W i (\tau )\,dy \right] d\tau} , \\
\\
e(t)-e_n(t) = (e(0)-e_n(0))+\int_{0}^{t} \left[s(\tau) \int_0^1\beta(\tau)W i(\tau )\,dy-
 s_n(\tau) \int_0^1\beta (\tau)W_{\G_n} i_n (\tau )\,dy \right] d\tau\\
~~~~~~~~~~~+\int_{0}^{t}\left[ \mu_n(\tau )e_n(\tau )-\mu (\tau )e (\tau)\right] d\tau, \\
\\
i(t)-i_n(t) = (i(0)-i_n(0))+\int_{0}^{t} \left[ (\mu (\tau )e (\tau )-\mu_n(\tau )e_n(\tau )) +(\gamma_n(\tau )i_n(\tau )-\gamma (\tau )i (\tau )) \right] d\tau, \\
r(t)-r_n(t) = (i(0)-i_n(0))+\int_{0}^{t} \left[ \gamma (\tau )i (\tau )- \gamma_n(\tau )i_n(\tau ) \right] d\tau.
\end{array}}}
\end{equation*}
For the first component, we have the following estimate,
\begin{equation*}
{\small{\begin{split}
&\vert s(t)-s_n(t)\vert \leq \vert s(0)-s_n(0)\vert+\int_{0}^{t} \left\vert s_n \int_0^1\beta_n W_{\G_n} i_n \,dy-
 s \int_0^1\beta W i \,dy \right\vert d\tau\\
&= \vert s(0)-s_n(0) \vert + \int_{0}^{t} \left\vert - s \int_0^1\beta W i \,dy +s_n \int_0^1 (\beta_n W_{\G_n} -\beta W) i \,dy +
s_n \int_0^1\beta W i_n \,dy \right\vert d\tau \\
&= \vert s(0)-s_n(0) \vert + \int_{0}^{t} \vert (s_n-s) \int_0^1\beta W i \,dy +s_n\int_0^1 (\beta_n W_{\G_n} -\beta W) i_n \,dy\\
&+s_n \int_0^1\beta W (i_n-i) dy \vert d\tau\\
&\leq \|s(0)-s_n(0) \|_{L^\infty} +\int_{0}^{t} \left[ K_w K_0 \left( \|s-s_n \|_{L^\infty} + \|i-i_n \|_{L^\infty} \right) +\|\beta W -\beta_n W_{\G_n}\|_{L^1} \right] d\tau.
\end{split}}}
\end{equation*}
Similarly, for the other components of~$(\mathbf{u}-\mathbf{u}_n)$ we obtain
\begin{equation*}
\begin{array}{l}
|e(t)-e_n(t)|
\leq \|e(0)-e_n(0) \|_{L^\infty}+\int_{0}^{t} \left[ K_w K_0 \left( \|s-s_n \|_{L^\infty} + \|i-i_n \|_{L^\infty} \right) \right.\\
\left.+\|\beta W -\beta_n W_{\G_n}\|_{L^1} \right] d\tau + \int_{0}^{t} K_0 \|e -e_n \|_{L^\infty} d\tau, \\ \\
|i(t)-i_n(t)| \leq \|i(0)-i_n(0) \|_{L^\infty} + \int_{0}^{t} K_0 \left( \|e-e_n \|_{L^\infty} + \|i-i_n \|_{L^\infty} \right) d\tau, \\ \\
|r(t)-r_n(t)| \leq \|r(0)-r_n(0) \|_{L^\infty} + \int_{0}^{t} K_0 \|i-i_n \|_{L^\infty} d\tau.
\end{array}
\end{equation*}
Let~$D(t)= \| \mathbf{u}(t) - \mathbf{u}_n(t) \|_{B_\infty}$,~$\hat{C}=\max \{ K_0,\,K_0K_W \}$, for~$t\in [0,T]$
\begin{equation*}
D(t) \leq \| \mathbf{u}(0) - \mathbf{u}_n(0) \|_{B_\infty} + T \| \beta W -\beta_n W_{\G_n}\|_{L^1} + \hat{C}\int_{0}^{t} D(\tau ) d\tau,
\end{equation*}
then, from Gronwall's Lemma,
\begin{equation*}
D(t) \leq \left[ \| \mathbf{u}(0) - \mathbf{u}_n(0) \|_{B_\infty} + T \| \beta W -\beta_n W_{\G_n}\|_{L^1} \right] (e^{\hat{C}t}-1).
\end{equation*}
The last term in the previous inequalities is bounded by the quantity~$(e^{\hat{C}T}-1)$, while~$\| \mathbf{u}(0) - \mathbf{u}_n(0) \|_{B_\infty} \rightarrow 0$, and
$\| \beta W -\beta_n W_{\G_n}\|_{L^1} \rightarrow 0$ as~$n\to+\infty$. Then~$D(t) \rightarrow 0$ uniformly with respect to time as~$n\to+\infty$.
\end{proof}
\begin{rem}
The conditions~\eqref{e:conv_tloc1} and \eqref{bound_deg_G} are satisfied under fairly reasonable assumptions. More precisely, assume that~$\{\G_n\}_{n\in \nn}$ is a sequence of graphs such that the corresponding sequence of graphons~$\{W_{\G_n}\}_{n \in \nn}$ is uniformly bounded in~$L^\infty$ and non-negative. Then there is a graphon~$W$ and a re-labelling of the vertices of~$\G_n$ such that~\eqref{e:conv_tloc1} holds. Note furthermore that any sequence of graphs with uniformly bounded adjacency matrices satisfies also \eqref{bound_deg_G}. One could also consider the case of time-dependent graphs, see~\cite{ADS} for the precise statement in this case.
\end{rem}
\subsection{Convergence for deterministic and random samplings\label{ss:samplings}}
Before introducing our next result, we make some heuristic considerations. Theorem~\ref{t:grafoni1} in the previous paragraph describes the continuum limit of~\eqref{SIR_G}. In particular, we can conclude that we have the convergence result for any sequence of graphs such that the corresponding sequence of graphons is uniformly bounded. In this sense, the statement of Theorem~\ref{t:grafoni1} goes ``from the sequence of graphs to the limit graphon".
\begin{figure}[t]
\centering
\begin{tabular}{c|c}
(a)\includegraphics[height=0.23\textwidth]{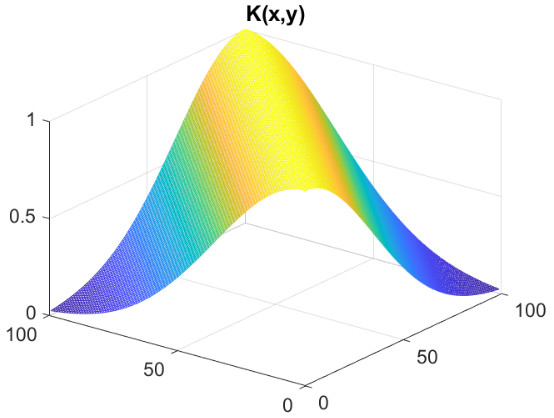}
&(b)\includegraphics[height=0.23\textwidth]{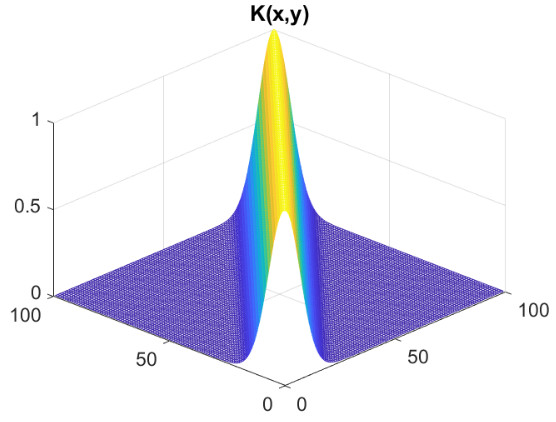}
\end{tabular}
\begin{tabular}{ccc}
\hline
\includegraphics[height=0.23\textwidth]{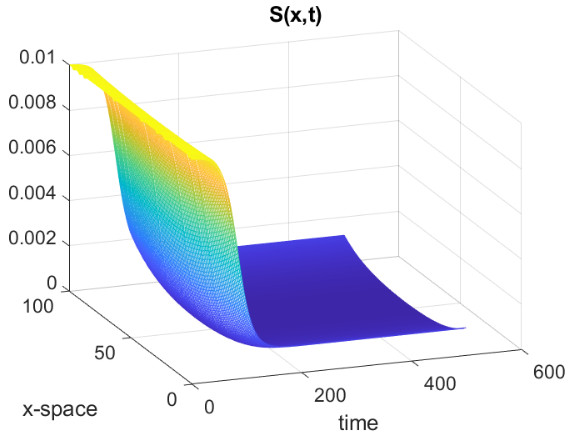}
&\includegraphics[height=0.23\textwidth]{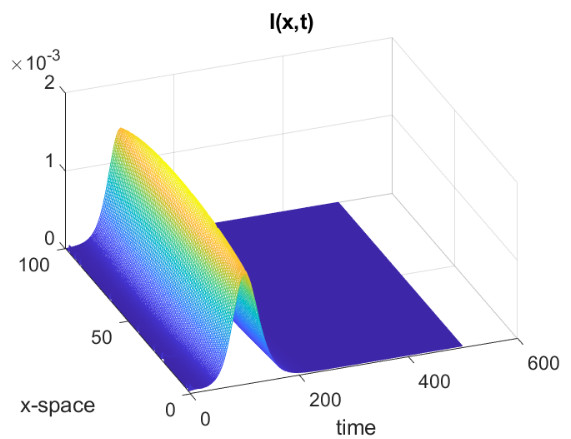}
&\includegraphics[height=0.23\textwidth]{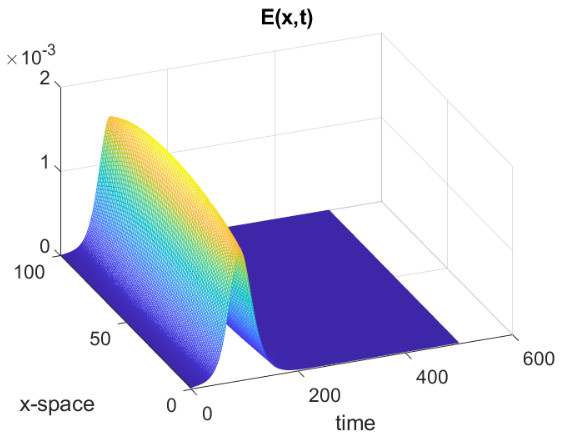}\\
(a) & &\\
\hline
\includegraphics[height=0.23\textwidth]{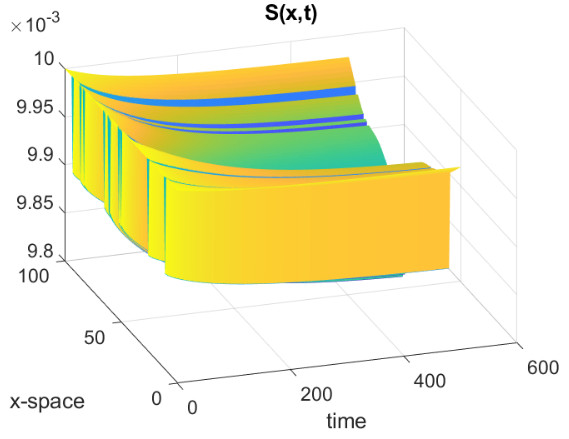}
&\includegraphics[height=0.23\textwidth]{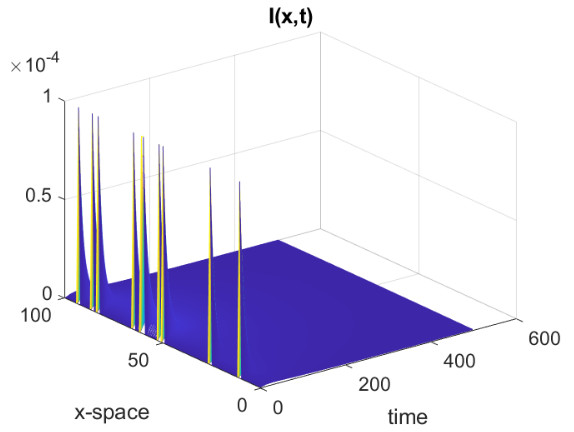}
&\includegraphics[height=0.23\textwidth]{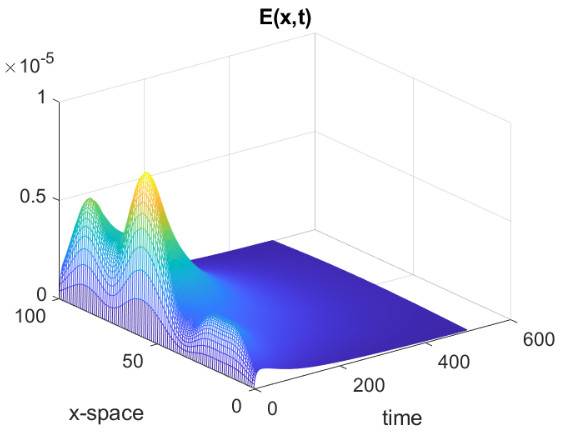}\\
(b) & &
\end{tabular}
\caption{A simple example: discretisation with~$n=100$ points: (a,b) selected Gaussian Kernel~$W_G(x,y)$, see \eqref{eq:graphon-gaussian}, and behaviour of the variables (b)~$S(x,t)$,~$I(x,t)$, and~$E(x,t)$ computed with a forward Euler scheme in time and composite mid-point rule for the integral.\label{fig:2D-SEIR-TEST}}
\end{figure}

We now investigate the reverse path ``from the limit graphon to the approximating sequence of graphs". More precisely, we fix a limit graphon~$W$ and we ask ourselves whether or not one can construct approximating sequences of graphs, initial data and coefficients in such a way that the corresponding solutions of~\eqref{SIR_G} provide a good approximation of the solution of~\eqref{SIR_W}. Our next result reports a possible answer to the above question.
\begin{theorem}
		\label{t:grafoni2}
		Let~$W\in L^\infty([0,1]^2)$ satisfy~$W \ge 0$ a.e. on~$[0, 1]^2$ and
		$$
		 \esssup_{x \in [0, 1]} \int_0^1 W(x, y) dy \leq K
		$$
		for some suitable constant~$K>0$. Assume that ~$\beta,\mu,\gamma\in L^\infty\left(\rr^+\times[0,1]\right)$ satisfy~$\beta, \mu, \gamma \ge 0$ and set
		\begin{equation} \label{e:betagamma}
 \beta_j^n(t):=\langle\beta (t, \cdot) \rangle_{I^n_j },\quad
 \mu_j^n(t):=\langle\mu (t, \cdot) \rangle_{ I^n_j}, \gamma_j^n(t):=\langle\gamma (t, \cdot) \rangle_{ I^n_j},
 \quad \text{for every~$j=1, \dots, n$},
		\end{equation}
		where~$\langle\beta \cdot \rangle_{I^n_j }$ denotes the average of the function on the set~$I^n_j : =[(j-1)n^{-1}, jn^{-1}[$. Assume furthermore that~$s_0, e_0, i_0,r_0\in L^\infty([0,1];[0,1])$ satisfy~$s_0+e_0+i_0+r_0=1$ a.e. in~$[0,1]$. Set
		\begin{equation*} \label{e:idata}
	 	s_{0,j}^n:=\langle s_0\rangle_{I^n_j}, \quad e_{0,j}^n:=\langle e_0\rangle_{I^n_j},
 \quad i_{0,j}^n:=\langle i_0\rangle_{I^n_j},
 \quad r_{0,j}^n:=\langle r_0\rangle_{I^n_j}.
		\end{equation*}
Then there is a deterministic and a random sampling of the graphon~$W$ such that, for any~$T>0$,
\begin{equation} \label{e:strongconv}
\left\|\normalfont(\svet^n,\evet^n,\ivet^n,\rvet^n)-(s,e,i,r)\right\|_{C^0(0,T;L^2([0,1],[0,1]^4))} \to0 \quad \text{ as }n\to+\infty\,,
\end{equation}
where~$(\normalfont\svet^n,\evet^n,\ivet^n,\rvet^n):\rr^+ \to[0,1]^4$ solves \eqref{SIR_G} and~$s,e,i,r:[0,T]\times[0,1]\to[0,1]$ satisfies \eqref{SIR_W} in the sense of distributions. The convergence result in in~\eqref{e:strongconv} holds almost surely in the case of random sampling.
\end{theorem}
Note that the convergence result in~\eqref{e:strongconv} is stronger than the one in~Theorem~\ref{t:grafoni1}. The proof of the above result and the explicit construction of the deterministic and random sampling is provided in~\cite{ADS} and is inspired by the argument in~\cite{Med19}. In the case of deterministic sampling, we can also handle the time-dependent case, that is~$W$ can also depend on the time variable. Also, through proper rescaling and suitable further assumptions, we can consider the case where~$W\in L^2 ([0,1]^2)$, see~\cite{ADS} for the technical details. To conclude, we point out that Theorem~\ref{t:grafoni2} may be relevant in the view of numerical considerations, and indeed in~\cite{ADS} we provide several numerical experiments that use the sampling constructed in the proof of Theorem~\ref{t:grafoni2}.
\begin{figure}[t]
\centering
\includegraphics[height=0.25\textwidth]{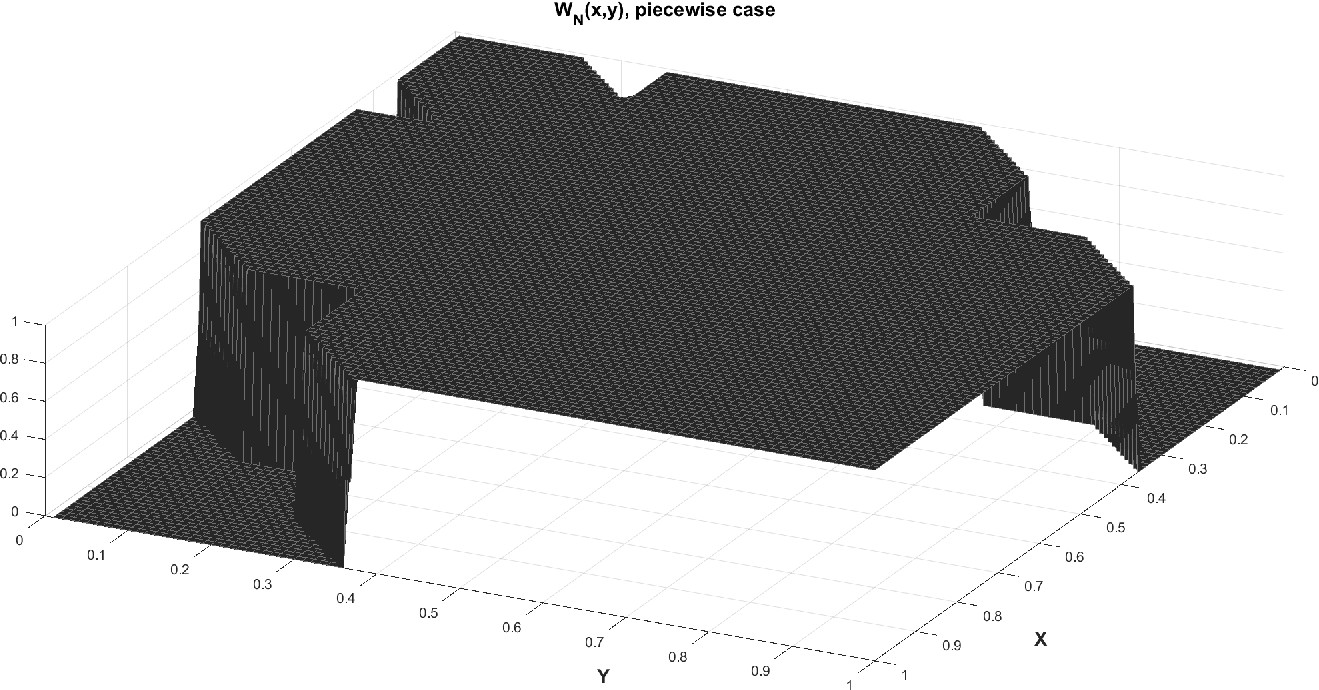}%
\includegraphics[height=0.25\textwidth]{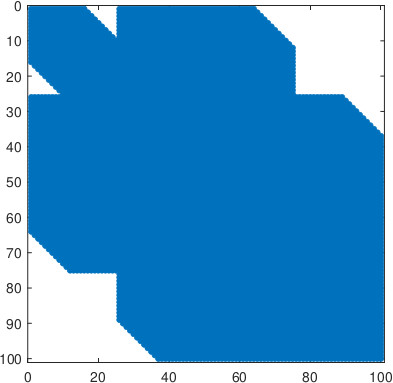}
\caption{A simple example of a piecewise graphon (or step graphon) for a block model of the epidemic spread.\label{fig:2D-BLOCK-MODEL}}
\end{figure}
\subsection{Semidiscretization of the SEIR model on a graphon}\label{sec:semid}
We consider the graphon operator
\begin{equation}\label{def:Tw}
(T_Wf)(x)=\int_0^1 W(x,y) f(y) dy, \qquad f \in L^2 ([0, 1]),
\end{equation}
which enjoys the following properties.
\begin{lemma}\label{l:compact}
Assume~$W \in L^2([0, 1]^2)$, then the map~$T_W$ defined by~\eqref{def:Tw} is a linear, bounded and compact operator from~$L^2([0, 1])$ to~$L^2 ([0, 1])$.
\end{lemma}
\begin{proof}
We first point out that, owing to the H\"older inequality,
\begin{equation*}
\begin{split}
&\| T_W f \|^2_{L^2([0, 1])} = \int_0^1 \left( \int_0^1 W(x, y) f(y) dy\right)^2 dx\\
&\leq\int_0^1 \| W(x, \cdot)\|^2 _{L^2([0, 1])}\| f\|^2 _{L^2([0, 1])} dx
= \| W\|^2 _{L^2([0, 1]^2)}\| f\|^2 _{L^2([0, 1])},
\end{split}
\end{equation*}
and this shows that the linear map~$T_W$ attains values in~$L^2([0, 1])$ and that it is a bounded (and henceforth continuous) operator.
\begin{figure}[t]
\centering
\begin{tabular}{c|c}
(a)\includegraphics[height=0.35\textwidth]{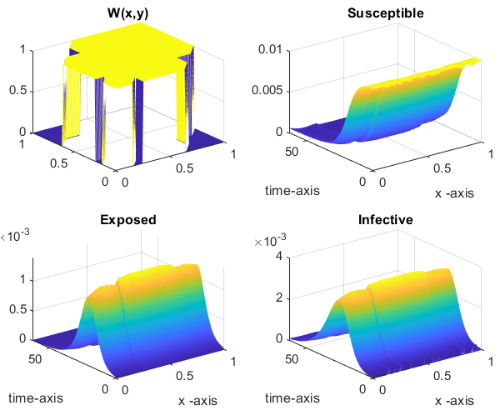}
&(b)\includegraphics[height=0.35\textwidth]{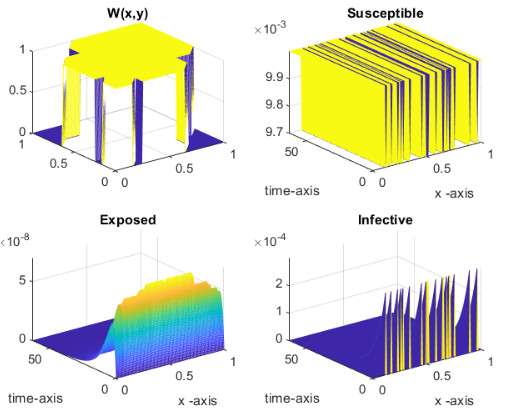}
\end{tabular}
\caption{A simple block model with a step graphon for which (a) the epidemic spreads and (b) does not spread.\label{fig:2D-SEIR-BLOCK-YES}}
\end{figure}

We are left to show that~$T_W$ maps bounded sets in relatively compact sets. Let us fix a bounded set~$B \subset L^2([0, 1])$ and a sequence~$\{ T_W f_n\}_{n \in \mathbb N}$ with~$\{ f_n \}_{n \in \mathbb N} \subset B$. Up to subsequences (which we do not relabel), the sequence~$f_n$ weakly converges in~$L^2([0, 1])$ to some limit function~$f$. To conclude, it suffices to show that~$T_W f_n$ strongly converges to~$T_W f$ in~$L^2([0, 1])$. First, we point out that~$(T_W f_n)(x)$ pointwise converges to~$(T_W f)(x)$ for a.e.~$x \in [0, 1]$. Next, we use again the H\"older inequality and get
$$
 \vert (T_W f )(x)\vert = \left\vert \int_0^1 W(x, y) f_n(y) dy \right\vert
 \leq
 \| W(x, \cdot)\|_{L^2([0, 1])}
 \| f_n \| _{L^2([0, 1])},
$$
for a.e.~$x \in [0, 1]$.
Since the sequence~$\{ f_n \}_{n \in \mathbb N}$ is bounded in~$L^2([0, 1])$, then the right hand side of the above inequality belongs to~$L^2([0, 1])$. By the Lebesgue Dominated Convergence Theorem, we conclude that~$T_W f_n$ strongly converges to~$T_W f$ in~$L^2([0, 1])$.
\end{proof}
To simulate the SEIR model on a graphon, see system (\ref{SIR_W}), we first consider a discretisation for the space variable~$x$ compatible with the result in the Theorem~\ref{t:grafoni2}. Then, starting from the uniform partition of the set~$[0,1]$ with the intervals~$\mathbbm{1}_{I_j^n}$, where~$I_j^n=[ (j-1)n^{-1}, jn^{-1}[$,~$j=1,\ldots,n$, we approximate the Graphon operator is defined by~\eqref{def:Tw} by the composite midpoint rule,
\begin{equation}\label{def:midpoint}
(\bar{T}_Wf)(x_k)=\frac{1}{n}\sum_{j=1}^{n} W(x_k,x_j) f(x_j),
\end{equation}
where~$x_k=(2k-1)/n$,~$k=1,\ldots,n$. The corresponding semi-discrete scheme is the following
\begin{equation}
		\label{SIR_W_semi}
		\begin{cases}
		\partial_t s(t,x_k)=-s(t,x_k)\frac{1}{n}\sum_{j=1}^{n} \beta(t,x_j)W(x_k,x_j)i(t,x_j), & \\
		\partial_t e(t,x_k)=s(t,x_k),
 \frac{1}{n}\sum_{j=1}^{n} \beta(t,x_j)W(x_k,x_j)i(t,x_j),\\
		\partial_ti(t,x_k)= \mu(t,x_k)e(t,x_k)-\gamma(t,x_k)i(t,x_k), & \\
 \partial_tr(t,x_k)=\gamma(t,x_k)i(t,x_k), & \\
		s(0,x_k)=s_0(x_k),\,e(0,x_k)=e_0(x_k),\,i(0,x_k)=i_0(x_k),\,r(0,x_k)=r_0(x_k),
		\end{cases}
	\end{equation}	
\mbox{$k=1,\ldots,n,$}. If the probability of disease transmission per contact parameter~$\beta$ is constant, the same results of the Lemma~\ref{Teo:seir1}, and Theorem~\ref{Teo:seir2} hold for the solution of the semi-discrete system (\ref{SIR_W_semi}). It follows that the dynamics of the transmission of the epidemic are determined in particular by the values of the eigenvalues of the operator~$\bar{T}_W$. We point out that~$W(x_k,x_j)$ represents a piecewise constant graphon, and then we focus on the spectral properties of the operator~$T_W$ corresponding to these last types of graphon. In the following, we assume~$W\in\cal{W}_0$, where~$\cal{W}_0$ is the set of all graphons attaining values in~$[0, 1]$.
%

Owing to Lemma~\ref{l:compact}, the operator~$T_W$ has a discrete spectrum, i.e. a countable multi-set of nonzero eigenvalues~$\{\lambda_1,\, \lambda_2,\ldots \}$ such that~$\lambda_n\rightarrow 0$. In particular, every nonzero eigenvalue has a finite multiplicity. If we consider the case of spreading epidemics in meta-populations, a natural community structure is the stochastic block model where each block represents a region, a city, a village, or other subsets of the population. The step graphons, related to a suitable partition of~$[0,1]$, with finite rank, encompass the properties of the stochastic block models. Now, the advantage provided by graphons lies in the fact that instead of considering a network with a large number of nodes, we consider a reduced model that maintains the same properties. A suitable graphon for a block models is a step function~$W$ such that a partition~$P=\{P_1,\, P_2,\ldots ,P_n\}$ of measurable sets of~$[0,1]$ exists, and~$W$ is constant on every product set~$P_i\times P_j$. Here we will consider the case of uniform partition with sub-intervals~$I_j^n$ defined above.

A graphon~$W$ can be used to generate a random graph using a sampling method, e.g. see~\cite{LO12}. Let~$N\in\mathbb{N}$, we say that the graph~$G$ is sampled from~$W$ if it is obtained through:
\begin{enumerate}
\item fixing deterministic latent variables~$\{u_i=\frac{i}{N} \}$,~$i=1,\, 2, \ldots N$;
\item taking~$N$ vertices~$\{1,\,2, \ldots ,\,N \}$ and randomly adding undirected edges between vertices~$i$ and~$j$ independently with probability
$W(u_i,u_j)$ for all~$i>j$.
\end{enumerate}
We can define the operator norm,
\begin{equation*}\label{eq:op:norm}
|||T_W||| : = \underset{\underset{\| f\|_{L^2}=1}{f\in L^2[0,1]}}{\sup}\,\| T_Wf \|_{L^2} .
\end{equation*}
Let~$T_{W_G}$ the graphon operator related to a graph sampled from~$W$, following ~\cite{APSS2020} it is possible to obtain an asymptotic estimate of the difference between the graphons in the operator norm. In particular, for a fixed value~$0< \nu < 1/e$ with probability~$1-\nu$,
\begin{equation}\label{est:prob:Tw}
|||T_W- T_{W_G}||| = O((\log{N}/N)^{1/2})\,\,\,N\rightarrow \infty .
\end{equation}
For graphons in~$\cal{W}_0$, the operator norm is equal to the largest, in modulus, eigenvalue, which is also real and positive, of the operator~$\lambda_1(T_W)$:
$|||T_W||| = \lambda_1(T_W)$. Then, using (\ref{est:prob:Tw}), with probability at least~$1-\nu$
\begin{equation*}\label{est:prob:lam}
\vert\lambda_1(T_W)-\lambda_1(T_{W_G})\vert=O((\log{N}/N)^{1/2})\,\,\,N\rightarrow \infty ,
\end{equation*}
and we have an approximation for the condition governing the transmission of the epidemic on the graph~$G$ obtained from~$W$. Moreover, in general, for a graphon~$W$, the eigenvalues form two sequences~$\mu_1(W)\geq \mu_2(W)\geq \ldots \geq 0$, and~$\mu^\prime_1(W)\leq \mu^\prime_2(W)\leq \ldots \leq 0$ converging to zero. Let~$\{ W_n \}_{n=1}^\infty$ be a sequence of uniformly bounded graphons, converging in the cut norm to a graphon~$W$. Then, see~\cite{BCLSV12}, for every~$i\geq 1$,
\begin{equation*}\label{eq:limit:lam}
\lim_{n\rightarrow\infty}\mu_i(W_n)=\mu_i(W),\,\,\,\lim_{n\rightarrow\infty}\mu^\prime_i(W_n)=\mu^\prime_i(W).
\end{equation*}
This implies that if a sequence of graphs converges in the cut norm to a graphon with a simple spectral characterisation by a few non-zero eigenvalues, as is the case for step graphon, then the sequence of graphs admits a simple low-dimensional spectral approximation. Furthermore, if the size of the graphs in the sequence is increasing then the low-dimensional approximation performs better. As a consequence, if a step graphon represents sufficiently well the social structure related to the spread of an epidemic, we can consider a low-dimensional model for an approximation of the behaviour of the phenomena on a large underlying network.

\subsection{Epidemic spread on a graphon: numerical examples\label{sec:ER}}
We conclude by showing some numerical simulations to illustrate some cases of the behaviour of a SEIR model on a (discrete) graphon. The numerical simulations were done using the semi-discrete system (\ref{SIR_W_semi}) coupled with the simple forward Euler scheme with a constant time step~$\tau$ for integration in time. We point out that more accurate and advanced schemes can be considered with an appropriate stability and convergence analysis. This type of extension and analysis goes beyond the scope of this work and we refer, for example, to~\cite{ADS} and our forthcoming work for more details. As a first example, we consider a Gaussian kernel like graphon~$W(x,y)$,
\begin{equation}\label{eq:graphon-gaussian}
W_G(x,y) = C_W e^{-\frac{1}{2}\left(\frac{(x-x_0)^2}{\sigma_x^2}+ \frac{(y-y_0)^2}{\sigma_y^2} \right)},
\end{equation}
where~$C_W$ is a suitable parameter in order to normalise the values of~$W(x,y)$, and~$(x_0,y_0)$ are the points on the diagonal of the square~$[0,1]\times [0,1]$. In all the simulations that we will show we have used a uniform partition of~$[0,1]$ with~$n=100$ subintervals. The parameters for the SEIR model are the following, from~\cite{GPV2020} and about the recent epidemic of SARS-CoV-2 in Italy,~$\beta=0.74$,~$mu=0.5$,~$\gamma=0.14$. In Fig.,~\ref{fig:2D-SEIR-TEST}, we represent the numerical results for two cases, controlled by a Gaussian graphon ($\sigma_x=\, \sigma_y = 0.5$): (a) spread, (b) no spread of the epidemic. In the second numerical example we consider a simple block model, the corresponding graphon is a piecewise function which is represented graphically in Fig.,~\ref{fig:2D-BLOCK-MODEL}. We perform a numerical simulation such that the condition~$\beta \lambda_1 (T_W) > \gamma$ is satisfied (compare with the Theorem~\ref{Teo:seir2}) and the epidemic spread as in the classical SEIR models (Fig.,~\ref{fig:2D-SEIR-BLOCK-YES}(a)). Finally, in Fig.,~\ref{fig:2D-SEIR-BLOCK-YES}(b) a numerical test with a block model is shown but without the spread of the epidemic.

\section{Conclusions and future work\label{sec:majhead}}
We have considered SEIR models defined on graphs, and discussed their basic features and in particular the spectral properties of the weighted adjacency matrix. We have also investigated the continuum limit obtained by considering graphs with a diverging number of nodes. This analysis provides a useful step in understanding the spread of diseases in populations with complex social structures. Possible further developments include the development of inferential methods for data on the emerging phase of epidemics, the extension of metapopulation models to more complex forms of human social structure, the development of metapopulation models able to capture the spatial population structure, the development of computationally efficient methods for calculating key epidemiological model quantities, and the integration of within- and between-host dynamics in models. In this work, we have also provided some preliminary experiments to support the possibility of approximating the dynamics of the epidemic for a large network by considering a low-dimensional approximation. More precisely, we focus on the case of the block model corresponding to a piecewise constant graphon. 
we point out that the extension to the case of directed graphs becomes essential to deal with realistic cases. However, this implies a revision of the theory of the graphons  considering non-symmetric kernels.

In future work, we will use this approximation to study control problems. A simple model in the case of a heterogeneous network was introduced in~\cite{ABN2021} while the linear case for a graphon has been done, for instance, in~\cite{GC2020}, we plan to analyse the non-linear case. Moreover, in some applications, it is necessary to estimate an underlying graphon to perform some network analysis. Some non parametric estimation methods have been proposed, and some are provably consistent. However, if certain useful features of the nodes (e.g. age, social group, health information) are available, it is possible to incorporate this source of information to help with the estimation, see e.g.~\cite{SWL2020}, using both the adjacency matrix and node features. 

Finally, a recent model for the spread of opinion on a graphon was introduced and analysed in~\cite{AN2022} using a~$L^2$ approach both for the convergence of the discrete system to the continuous one and for the analysis of the continuous problem. In this way, analytical solutions are available in the piecewise graphon case. We want to extend this type of approach also to the case of epidemiological models.\\

\textbf{Acknowledgements}
We thank the Reviewers for their thorough review and constructive comments, which helped us to improve the technical part and presentation of the revised paper. The authors would also like to thank Simone Dovetta (POLITO), Laura Spinolo (CNR-IMATI), and several colleagues that have given us helpful suggestions and advice; in particular, the COVID-19 modelling group organised by the CNR-IMATI. G. Naldi acknowledges the support of the project DIT.AD021.001 (Numerical Analysis and Scientific Computing) of IMATI-CNR (Italy).
\textbf{Conflict of interest statement} On behalf of all authors, the corresponding author states that there is no conflict of interest.

\end{document}